\documentclass[12pt]{amsart}
\usepackage{amsmath}
\usepackage{amssymb}
\usepackage{bm}
\usepackage{graphicx}
\usepackage{verbatim}
\usepackage{enumitem}
\usepackage[pdftex,pdfstartview=FitH,pdfborderstyle={/S/B/W 1}]{hyperref}
\usepackage[nobysame,alphabetic,initials,msc-links]{amsrefs}

\newtheorem{Lemma}{Lemma}[section]
\newtheorem{Theorem}[Lemma]{Theorem}
\newtheorem{Proposition}[Lemma]{Proposition}
\newtheorem{Corollary}[Lemma]{Corollary}

\newtheorem{Definition}[Lemma]{Definition}

\numberwithin{equation}{section}

\newcommand {\C} {{\mathbb C}}

\DeclareMathOperator {\supp} {supp}

\newcommand {\Q} {{\mathbb Q}}
\newcommand {\R} {{\mathbb R}}

\renewcommand {\H} {{\mathbb H}}
\newcommand {\Z} {{\mathbb Z}}
\newcommand {\absolute}[1] {\left| {#1} \right|}
\newcommand {\norm}[1] {\left\| {#1} \right\|}
\newcommand {\vol}{\mathrm{vol}}
\newcommand {\area}{\mathrm{area}}
\DeclareMathOperator {\PSL} {PSL}
\DeclareMathOperator {\SL} {SL}
\DeclareMathOperator {\SO} {SO}
\DeclareMathOperator {\diam} {diam}

\begin{document}

\title[Joint Quasimodes and QUE]{Joint Quasimodes, Positive Entropy, and Quantum Unique Ergodicity}
\author{Shimon Brooks and Elon Lindenstrauss}
\thanks{S.B. was partially supported by NSF grant DMS-1101596.  E.L. was
supported by the ERC, NSF grant DMS-0800345, and ISF grant 983/09.}
\begin{abstract}
We  
 study joint quasimodes of the Laplacian and one Hecke operator on compact congruence surfaces, and give conditions on the orders of the quasimodes that guarantee positive entropy on almost every ergodic component of the corresponding semiclassical measures.  Together with the measure classification result of \cite{Lin}, this implies Quantum Unique Ergodicity for such functions.  
Our result is optimal with respect to the dimension of the space from which the quasi-mode is constructed.

 We also study equidistribution for sequences of joint quasimodes of the two partial Laplacians on compact irreducible quotients of $\mathbb{H}\times\mathbb{H}$. 
\end{abstract}

\maketitle

\section{Introduction}
The Quantum Unique Ergodicity (QUE) Conjecture of Rudnick-Sarnak \cite{RS} states that  eigenfunctions of the Laplacian on Riemannian manifolds of negative sectional curvature become 
equidistributed in the high-energy limit.  Although there exist so-called ``toy models" of quantum chaos that do not exhibit this behavior (see eg. \cites{FNDB, ANb, KelPert}), it has been suggested that large degeneracies of the quantum propagator may be responsible for some of
these phenomena
(see eg. \cite{SarnakProgress}).  Since the Laplacian on a surface of negative curvature is not expected to have large degeneracies, one can explore this aspect and introduce ``degeneracies" by considering quasimodes, or approximate eigenfunctions, in place of true eigenfunctions--- relaxing the order of approximation to true eigenfunctions yields larger spaces of quasimodes, mimicking higher-dimensional eigenspaces.  Studying the properties of such quasimodes--- and, especially, the effect on equidistribution of varying the order of approximation--- can help shed light on the overall role of spectral degeneracies in the theory.

One case of QUE that has been successfully resolved is the so-called Arithmetic Quantum Unique Ergodicity, where one considers congruence surfaces that carry additional number-theoretic structure, and admit Hecke operators that exploit these symmetries.  Since these operators commute with each other and with the Laplacian, it is natural to consider joint eigenfunctions of the Laplacian and Hecke operators.  The dimension of the joint eigenspace of the full Hecke algebra is well understood, and is one unless there is an obvious symmetry; in any case this dimension is bounded by a constant depending only on $\Gamma$, and hence degeneracies in the spectrum play no role is this case. For compact congruence surfaces, QUE for such joint eigenfunctions was proved by the second-named author \cite{Lin}, relying on earlier work with Bourgain \cite{BorLin}. In the non-compact case the results of \cite{Lin} are somewhat weaker, but this has been rectified by Soundararajan in \cite{Sound}; some higher-rank cases were studied in \cites{LinHxH, Lior-Akshay, Lior-Akshay2, AnanSilberman}.  
 
Here, we apply techniques developed in the context of our work \cite{meElon} on eigenfunctions of large graphs to the question of Quantum Unique Ergodicity for joint {\em quasimodes}.  We define an $\omega(r)$-{\bf quasimode with approximate parameter} $r$ to be a function $\psi$ satisfying
$$||(\Delta + (\frac{1}{4}+r^2))\psi||_2 \leq r \omega(r)||\psi||_2$$
The  factor of $r$ in our definition comes from the fact that $r$ is essentially the square-root of the Laplace eigenvalue.  

We will be interested in $o(1)$-quasimodes, by which we mean $\tilde{\omega}(r)$-quasimodes for some fixed $\tilde{\omega}(r)$ tending to 0 as $r\to\infty$; we  allow this decay to be arbitrarily slow.  Denoting by $S_\omega(r)$ the space spanned by eigenfunctions of spectral parameter in $[r-\omega(r), r+\omega(r)]$, our $o(1)$-quasimodes should be thought of as essentially belonging to $S_\omega(r)$ for some $\omega(r)\to 0$.  In fact,

\begin{Lemma}\label{quasimodes}
Given a fixed function $\tilde{\omega}(r)\to 0$, and a sequence $\{\tilde{\psi}_j\}_{j=1}^\infty$ of $\tilde{\omega}(r_j)$-quasimodes with approximate eigenvalue $r_j\to\infty$, there exists a function $\omega(r)\searrow 0$ and a sequence $\{{\psi}_j \in S_\omega(r_j)\}_{j=1}^\infty$ such that
$$||\psi - \tilde{\psi}||_2 \to 0 \quad \text{ as } \quad j\to\infty$$
\end{Lemma}
This means that the microlocal lifts of $\psi$ and $\tilde{\psi}$ (see section~\ref{ml lifts}) have the same weak-* limit points.  Therefore, for our purposes, we can assume without loss of generality that our $o(1)$-quasimodes actually belong to $S_\omega(r_j)$ for some $\omega(r)\searrow 0$.

{\em Proof:}
Consider $\tilde\psi_j^\perp = \tilde\psi_j - \Pi_{S_\omega(r_j)}\tilde\psi_j$, the projection of $\tilde\psi$ to the orthogonal complement of $S_\omega(r_j)$, and decompose $\tilde\psi_j^\perp = \sum_i c_j(\phi_i)\phi_i$ into an orthonormal basis of eigenfunctions.  Since each  component $\phi_i$ of $\tilde\psi_j^\perp$ has spectral parameter outside the interval $[r_j-\omega(r_j), r_j+\omega(r_j)]$, we have
$$||(\Delta + (\frac{1}{4}+r_j^2))\phi_i||_2 > 2\omega(r_j) r_j - \omega(r_j)^2$$
and therefore
\begin{eqnarray*}
||(\Delta + (\frac{1}{4}+r_j^2))\tilde{\psi}_j||_2 & \geq & ||(\Delta + (\frac{1}{4}+r_j^2))\tilde{\psi}_j^\perp||_2\\
& \gtrsim & \sqrt{\sum_i|c_j(\phi_i)|^2}\omega(r_j)r_j
\end{eqnarray*}
Choosing $\omega(r)$ to decay sufficiently slowly that $\tilde{\omega}(r) = o(\omega(r))$, this contradicts the $\tilde{\omega}(r)$-quasimode hypothesis unless $\sum_i |c_j(\phi_i)|^2 \to 0$ (and hence $||\tilde\psi_j^\perp||_2\to0$) as $j\to\infty$.  $\Box$

One expects that the space $S_\omega(r)$ should have dimension proportional to $r\omega(r)$; in fact, estimating the error term in Weyl's Law shows that this holds for windows of size $\omega(r)\geq C/\log{r}$, where $C$ is a constant depending on the manifold \cite[\S4]{SarnakHyperbolic}.  The dimension of $S_\omega(r)$ is not known for small windows $\omega(r)$--- indeed, this is the very difficult problem of bounding spectral multiplicities--- but we will content ourselves to deal with $o(1)$-quasimodes, for which it is known that the spaces are of dimension $o(r)$.  Such quasimodes can be extremely crude, and without further assumptions one cannot expect meaningful statements.  However, the situation is radically different when there is additional structure present that one can exploit.

Our first result concerns certain compact hyperbolic surfaces $\Gamma\backslash\mathbb{H}$ of arithmetic congruence type.
One can consider more general $\Gamma$, but for concreteness and simplicity we restrict to the following situation.
Let $H$ be a quaternion division algebra over $\mathbb{Q}$, split over $\mathbb{R}$, and $R$ a maximal order in $H$.  Fix an isomorphism $ \iota : H(\mathbb{R}) \overset\sim\to \operatorname{Mat}_2(\mathbb{R})$.
For $\alpha\in R$ of positive norm $n(\alpha)$, we write $\underline{\alpha}$ for the corresponding element of $\PSL_2(\mathbb{R})$.  Set $\Gamma$ to be the image in $\PSL_2(\mathbb{R})$ of the subgroup of norm $1$ elements of $R$.  As is well known, $\Gamma$ is discrete and co-compact in $\PSL_2(\mathbb{R})$, and the quotient $X=\Gamma\backslash \PSL_2(\mathbb{R})$ can be identified with the unit cotangent bundle of a compact hyperbolic surface $M=\Gamma\backslash \mathbb{H}$.

Write $R(m)$ for the set of elements of $R$ of norm $m$, and define the Hecke operator
$$T_m: f(x) \mapsto \frac{1}{\sqrt{m}}\sum_{\alpha\in R(1)\backslash R(m)} f(\underline{\alpha}x)$$
as the operator averaging over the Hecke points
$$T_m(x) = \{\underline{\alpha}x : \alpha\in R(1)\backslash R(m)\};$$
note that these formulas make sense both for $x\in \Gamma\backslash \H$ and $x\in \Gamma\backslash \PSL_2(\R)$.

Our methods are reliant only on the Laplacian and a single Hecke operator, and so we will be interested in the case where $m=p^{k}$ are powers of a fixed prime $p$.  It is well known that $T_{p^k}$ is a polynomial in $T_p$; so in particular, eigenfunctions of $T_p$ are eigenfunctions of all $T_{p^k}$.  It is known that for all but finitely many primes, the points $T_{p^k}(x)$ form a $p+1$-regular tree as $k$ runs from $0$ to $\infty$; we will always assume that $p$ is such a prime.  We denote by $S_{p^k}$ the sphere of radius $k$ in this tree, given by Hecke points corresponding to the primitive elements of $R$ of norm $p^k$.

Since each $\underline{\alpha}$ acts by isometries on $\Gamma\backslash\mathbb{H}$, the operator $T_p$ commutes with $\Delta$.  The spectrum of $T_p$ on $L^2(\Gamma\backslash\mathbb{H})$ lies in $[-\frac{p+1}{\sqrt{p}}, \frac{p+1}{\sqrt{p}}]$, and since $T_p$ commutes with $\Delta$, there is an orthonormal basis of $L^2(\Gamma\backslash\mathbb{H})$ consisting of joint eigenfunctions for $T_p$ and $\Delta$.  An $\omega$-quasimode for $T_p$ of approximate eigenvalue $\lambda$ is a function $\psi$ satisfying $||(T_p - \lambda)\psi||_2 \leq \omega||\psi||_2$; we call $\{\psi_j\}_{j=1}^\infty$ a {\bf sequence of $o(1)$-quasimodes for $T_p$} if each $\psi_j$ is an $\omega_j$-quasimode for $T_p$, and $\omega_j\to 0$ as $j\to\infty$.  As above in Lemma~\ref{quasimodes}, for our purposes we may as well assume that each $\omega_j$-quasimode is a linear combination of eigenfunctions whose eigenvalues lie in intervals of the form $[\lambda_j-\omega_j, \lambda_j+\omega_j]$, for some sequence $\omega_j\to 0$.  

For any sequence $\phi_j$ of $o(1)$-quasimodes  of the Laplacian $\Delta$ on $M$, normalized by $||\phi_j||_2=1$, there exists a measure $\mu_j$ on $S^*M$ (which we view as a probability measure on $\Gamma \backslash \PSL_2(\mathbb{R})$), called the {\bf microlocal lift} of $\phi_j$.  This measure is asymptotically invariant under the geodesic flow as the approximate eigenvalue of $\phi_j$ tends to infinity, and its projection to $M$ is asymptotic to $|\phi_j|^2d\area$.  Since the space of probability measures on $M$ is weak-* compact, we consider weak-* limit points of these microlocal lifts, called {\bf semiclassical measures} or {\bf quantum limits}.  The QUE problem asks whether such a limit point must be the uniform measure on $S^*M$.
We shall use a version of the construction of these lifts due to Wolpert \cite{Wolpert} (this version is the one used in \cite{LinHxH,Lin}), where $\mu_j=|\Phi_j|^2d\vol$ for suitably chosen $\Phi_j \in L^2(S^*M)$; see Section~\ref{ml lifts}.
The construction satisfies that $\Phi_j$ is an eigenfunction of $T_p$ when $\phi_j$ is, of the same eigenvalue; more generally, applying this to each spectral component, $\Phi_j$ is an $\omega_j$-quasimode for $T_p$ whenever $\phi_j$ is, with the same approximate eigenvalue.
Since $\Delta$ commutes with $T_p$, we may consider sequences $\{\phi_j\}$ of joint $o(1)$-quasimodes, whereby the corresponding $\{\Phi_j\}$ are still $o(1)$-quasimodes of $T_p$ as well.

We recall the following key definition from \cite{Lin}:

\begin{Definition} A measure $\mu$ on $S^*M=\Gamma \backslash \PSL_2 (\R)$ is said to be \textbf{$(\mathbf) {T_p}$-recurrent} if for any subset $A \subset S^*M$ of positive $\mu$-measure, for $\mu$-almost every $x \in A$ there is a sequence $k_i \to \infty$ for which $S _ {p^{k_i}} (x) \cap A$ is nonempty.
\end{Definition}

As we show in Lemma~\ref{Hecke recurrence by propagation lemma}, any weak-* limit point of the $\mu_j$ is $T_p$-recurrent.  The main innovation in this paper is the following result regarding these limit measures:

\begin{Theorem}\label{main}
Let $p$ be a prime (outside the finite set of bad primes for $M$), and let $\{\phi_j\}_{j=1}^\infty$ be a sequence of $L^2$-normalized joint $o(1)$-quasimodes of $\Delta$ and $T_p$ on $M$.  Then any weak-* limit point $\mu$ of the microlocal lifts $\mu_j$ has positive entropy on almost every ergodic component.
\end{Theorem}

In view of the measure classification results of \cite{Lin}, this theorem,  together with the geodesic-flow invariance of quantum limits proved in Lemma~\ref{geo flow invariance}, and $T_p$-recurrence which we establish in Lemma~\ref{Hecke recurrence by propagation lemma} imply the following:
\begin{Corollary}\label{QUE}
Let $\{\phi_j\}$ as above be a sequence of joint $o(1)$-quasimodes of $\Delta$ and $T_p$.   Then the sequence $\mu_j$ converges weak-* to Liouville measure on $S^*M$.
\end{Corollary}

Here, we have reduced the assumptions to a bare minimum:  we assume only that our functions are joint $o(1)$-quasimodes for the Laplacian and {\em one} Hecke operator.  In particular, this shows that the full Hecke algebra is not needed to establish Arithmetic QUE, as one Hecke operator will suffice.  Note that even without any Hecke operators, Anantharaman \cite{Anan} has shown--- for general negatively curved compact manifolds--- that a quantum limit corresponding to Laplacian eigenfunctions 
has positive entropy, and these techniques extend to $O(\frac{1}{\log{r}})$-quasimodes).
Anantharaman's result has been further sharpened in her joint work with Nonnenmacher \cite{AN} (see also \cite{AKN}). With respect to eigenfunctions, while the {\em overall} entropy bounds obtained by our methods are weaker than those of \cite{AN}, Theorem~\ref{main} is stronger in that it gives information on {\em almost all ergodic components}. In particular, the results of \cites{Anan,AN,AKN} do not rule out that a positive proportion of the mass of a quantum limit is supported on a single periodic trajectory of the geodesic flow.

Our results go further by allowing $o(1)$-quasimodes--- which may be taken from rather large subspaces whose dimension is only bounded by $o(r)$.  In the paper \cite{localized_example} the first named author shows that these subspaces are ``optimally degenerate" for QUE, in the sense that given any sequence of $c r_j$-dimensional subspaces $\tilde{S}_c(r_j)$ of the spaces $S_C(r_j)$, there exists a sequence $\{\psi_j\in \tilde{S}_c(r_j)\}_{j=1}^\infty$ such that the corresponding microlocal lifts do not converge to Liouville measure on $S^*M$--- indeed, any weak-* limit point of these microlocal lifts must concentrate a positive proportion of its mass on a codimension $1$ subset!  Thus, our joint $o(1)$-quasimodes form spaces of largest-possible dimension that can satisfy QUE.

It should be remarked that these subspaces are considerably larger than the spaces of Laplace-quasimodes that are expected to satisfy QUE without any Hecke assumption.  This is a testament to the rigidity imposed by the additional structure of the Hecke correspondence, as already apparent in \cite{Lin}.

It is worth remarking here that the phrase ``joint $o(1)$-quasimodes" can be misleading:  the two orders of approximation are not necessarily the same, and serve completely separate purposes; the approximation to the Laplace eigenvalue is used to guarantee asymptotic invariance under the geodesic flow, while the approximation to the Hecke eigenvalue is used to establish positive entropy on almost every ergodic component and $T_p$-recurrence.

Our methods also apply to the case of $M=\Gamma\backslash \mathbb{H}\times\mathbb{H}$, with $\Gamma$ a co-compact, irreducible lattice in $\PSL(2,\mathbb{R})\times \PSL(2,\mathbb{R})$.
Here we do not assume any Hecke structure\footnote {By Margulis' Arithmeticity Theorem such lattices are necessarily arithmetic, though it is not known if they are necessarily of congruence type (which is necessary for the existence of Hecke operators with good properties).}; instead, we take the sequence $\{\phi_j\}$ to consist of joint $o(1)$-quasimodes of the two partial Laplacians, each on the respective copy of $\mathbb{H}$.  The Laplacian on $M$ is the sum of the two partial Laplacians, and so the large eigenvalue limit for the Laplacian entails at least one of the two partial eigenvalues going to infinity (after passing to a subsequence, if necessary).  At present, we have been able to apply our methods only to the case where one partial eigenvalue (say the second) remains  bounded, and the other tends to $\infty$.

By the arguments of \cite{LinHxH} and Section~\ref{HxH lift}, this means that the microlocal lift to $\Gamma\backslash \PSL(2,\mathbb{R})\times\mathbb{H}$ becomes invariant under the action of the diagonal subgroup $A$ of $\PSL(2,\mathbb{R})$ acting on the first coordinate.  By applying the methods of Theorem~\ref{main} to use the foliation given by varying the second coordinate in an analogous way to use of the Hecke Correspondence $T_p$ for congruence surfaces, we are able to prove that any quantum limit of such a sequence must also carry positive entropy on a.e. ergodic component (with respect to the $A$-action on the first coordinate).  Thus, \cite[Thm.\ 1.1]{Lin} in conjunction with the relevant recurrence property proved in Lemma~\ref{hyperbolic recurrence by propagation lemma} again may be used to show equidistribution of $|\phi_j|^2d\vol$ as well as their microlocal lifts:
\begin{Theorem}\label{QUE on HxH}
Let $\{ \phi_j \}$ be a sequence of joint $o(1)$-quasimodes of $\Delta_1$ and $\Delta _ 2$ on $\Gamma \backslash \H \times \H$ with approximate spectral parameters $r^1_j,r^2_j$ and where $\Gamma$ is an irreducible cocompact lattice in $\PSL (2, \R) \times \PSL (2, \R)$. Assume $r^1_j \to \infty$ and $r^2_j$ bounded.  Then the sequence of lifts $\mu_j$ of $\absolute {\phi _ j}^2 d \vol$ to $\Gamma \backslash \PSL (2, \R) \times \H$ converges weak-* to the uniform measure on $\Gamma \backslash \PSL (2, \R) \times \H$.
\end{Theorem}

\noindent The argument is analogous to the one presented here for the rank-one arithmetic case in Theorem~\ref{main}, and we present the necessary modifications of the argument in section~\ref{HxH}.

The results of \cite{BorLin}, along with an appropriate construction of microlocal lifts,  were generalized by Silberman and Venkatesh \cite{Lior-Akshay, Lior-Akshay2} who, using the measure classification results of \cite{EKL}, were able to extend the QUE results of \cite{Lin} to quotients of more general  symmetric spaces (they also had to develop an appropriate microlocal lift). It is likely possible to extend the techniques of this paper to their context.  We also hope that our work may be extended to the case of $\Gamma\backslash \mathbb{H}\times\mathbb{H}$ where both partial eigenvalues are growing to $\infty$.

Regarding finite volume arithmetic surfaces such as $SL(2,\mathbb{Z})\backslash\mathbb{H}$, it was shown in \cite{Lin} that any quantum limit
has to be a scalar multiple of the Liouville measure--- though not necessarily with the right scalar--- and similar results can be provided by  our techniques using a single Hecke operator. Recently, Soundararajan \cite{Sound} has given an elegant argument that settles this escape of mass problem for joint eigenfunctions of the full Hecke algebra and (in view of the results of \cite{Lin}) shows that the only quantum limit is the normalized Liouville measure. An interesting open question is whether our $p$-adic wave equation techniques can be used to rule out escape of mass using a single Hecke operator. We also mention that Holowinsky and Soundararajan \cite{Holowinsky-Sound} have recently developed an alternative approach to establishing Arithmetic Quantum Unique Ergodicity for joint eigenfunctions of all Hecke operators. This approach \emph{requires} a cusp, and is only applicable in cases where the Ramanujan Conjecture holds; this conjecture is open for the Hecke-Maass forms, but has been established by Deligne for holomorphic cusp forms --- a case which our approach does not handle.

\subsection*{Acknowledgments}

We thank Peter Sarnak for helpful discussions, and Shahar Mozes for helpful suggestions that simplified the proof of Lemma \ref{HxH disjoint}.

\section{Microlocal Lifts of Quasimodes}\label{ml lifts}

\subsection{Some Harmonic Analysis on $\PSL(2,\mathbb{R})$}
We begin by reviewing some harmonic analysis on $\PSL(2,\mathbb{R})$ that we will need.  
Throughout, we write $X=\Gamma\backslash \PSL(2,\mathbb{R})$ and $M=\Gamma\backslash\mathbb{H} = \Gamma\backslash \PSL(2,\mathbb{R})/K$, where $K=SO(2)$ is the maximal compact subgroup.

Fix an orthonormal basis $\{\phi_j\}$ of $L^2(M)$ consisting of Laplace eigenfunctions, which we can take to be real-valued for simplicity.  Each eigenfunction generates, under right translations, an irreducible representation $V_{j} = \overline{\{\phi_j(xg^{-1}) : g\in \PSL(2,\mathbb{R})\}}$ of $\PSL(2,\mathbb{R})$, which span a dense subspace of $L^2(X)$.  

We distinguish the pairwise orthogonal {\bf weight spaces} $A_{2n}$  in each representation, consisting of those functions satisfying $f(xk_\theta) = e^{ i 2n\theta}f(x)$ for all $k_\theta = \begin{pmatrix} \cos{\theta} & \sin{\theta}\\ -\sin{\theta} & \cos{\theta}\end{pmatrix} \in K$ and $x\in X$.  The weight spaces together span a dense subspace of $V_{j}$.  Each is one-dimensional in $V_{j}$, spanned by $\phi^{(j)}_{2n}$ where
\begin{eqnarray*}
\phi_0^{(j)} & = & \phi_j \in A_0\\
(ir_j + \frac{1}{2}+n)\phi_{2n+2}^{(j)} & = & E^+\phi_{2n}^{(j)}\\ 
(ir_j + \frac{1}{2} -n)\phi_{2n-2}^{(j)} & = & E^-\phi_{2n}^{(j)}
\end{eqnarray*}
Here $E^+$ and $E^-$ are the {\bf raising} and {\bf lowering operators}, first-order differential operators corresponding to $\begin{pmatrix} 1 & i\\ i & -1\end{pmatrix}\in \mathfrak{sl}(2,\mathbb{C})$ and $\begin{pmatrix} 1 & -i\\ -i & -1 \end{pmatrix}\in \mathfrak{sl}(2,\mathbb{C})$ in the complexified Lie algebra.  
The normalization is such that each $\phi_{2n}^{(j)}$ is a unit vector.
Continuing with the identification of the elements in $\sl (2, \C)$ with first order invariant differential operators on $X$, we also make use of the following operators
\begin{eqnarray*}
H & = & \begin{pmatrix} 1&0\\0 & -1\end{pmatrix}\\
W & = & \begin{pmatrix} 0 & -1\\ 1 & 0\end{pmatrix}\\
X^+ & = & \begin{pmatrix} 0 & 1\\ 0 & 0 \end{pmatrix}
\end{eqnarray*}
note that $H$ is the derivative in the geodesic-flow direction, $W$ is the derivative in the fibre $K$ direction, and $X^+$ is the derivative in the stable horocycle direction. We will also need the second order operator
\begin{equation*}
\Omega = {\tfrac{1}{2}} (E ^ + E ^ - + E ^ - E ^ +) + \frac {W ^2 }{ 4}
\end{equation*}
which commutes with all the invariant differential operators on $X$ and agrees with $\Delta$ on the space of $K$-invariant functions.

It is easy to check that the distribution
$$\Phi_\infty^{(j)} = \sum_{n=-\infty}^{\infty} \phi_{2n}^{(j)}$$
satisfies $H \Phi_\infty^{(j)} = (ir_j-\frac{1}{2})\Phi_\infty^{(j)}$ and $X^+\Phi_\infty^{(j)}= 0$.  We also record the identities
\begin{eqnarray}
H & = & \frac{1}{2}(E^+ + E^-)\nonumber\\
E^+\Phi_\infty^{(j)} & = & (ir_j -\frac{1}{2} - \frac{i}{2}W)\Phi_\infty^{(j)}\nonumber\\
E^-\Phi_\infty^{(j)} & = & (ir_j - \frac{1}{2}+\frac{i}{2}W)\Phi_\infty^{(j)}\nonumber\\
E^+E^-\Phi_\infty^{(j)} & = & -(r_j^2+\frac{1}{4})\Phi_\infty^{(j)} + (\frac{1}{4}W^2 -\frac{i}{2}W)\Phi_\infty^{(j)}\nonumber\\
E^+E^-\phi_0^{(j)} & = & -(r_j^2+\frac{1}{4})\phi_0^{(j)} =  \Delta \phi_0^{(j)}\label{identities}
\end{eqnarray}

\subsection{Construction of the Microlocal Lift}
For any $\omega(r)$-quasimode $\phi$,  with approximate Laplace eigenvalue $\frac{1}{4}+r^2$, we seek a smooth function $\Phi$ on $\Gamma\backslash \PSL(2,\mathbb{R})$ satisfying:\label{Greek Phi properties}
\begin{enumerate}[leftmargin=*,label=\textup{($\Phi$\arabic*)}]
\item{The measure $|\Phi|^2d\vol$ is asymptotic to the distribution defined by 
$$f\mapsto \langle Op(f)\phi, \phi\rangle_{L^2(M)} = \langle f\Phi_\infty, \phi\rangle_{L^2(X)}$$
In particular, the projection of the measure $|\Phi|^2d\vol$ on $X$ to $M$ is asymptotic to $|\phi|^2d\area$; i.e., for any smooth $\tilde{f}\in C^\infty(M)$, we have
$$\int_{X} \tilde{f} |\Phi|^2d\vol \sim \int_{M} \tilde{f} |\phi|^2 d\area$$
as $r\to\infty$.  It is in this sense that $|\Phi|^2d\vol$ ``lifts" the measure $|\phi|^2d\area$ to $X$.}
\item{The measure $|\Phi|^2d\vol$ is asymptotically invariant under the geodesic flow; that is, for any smooth $f \in C^\infty(X)$, we have
$$\int_{X} H\!f\, |\Phi|^2d\vol \to 0 \qquad \text{as $r \to \infty$} $$
}
\item {$\Phi$ will be a $T_p$-eigenfunction (or an $\omega$-quasimode for $T_p$) whenever $\phi$ is.
}

\end{enumerate}

\noindent For the remainder of the section, set
$$I_\phi(f) := \langle f \Phi_\infty, \phi\rangle = \lim_{N\to\infty} \left\langle f \sum_{n=-N}^N \phi_{2n}, \phi_0\right\rangle.$$
Note that this limit is purely formal for $K$-finite $f$, by orthogonality of the weight spaces, and since these $K$-finite functions are dense in the space of smooth functions, we can restrict our attention to these.  We denote by $A_{2n}$ the $n$-th weight space, consisting of smooth functions that transform via $f(xk_\theta) = e^{2i n\theta} f(x)$ for all $x\in X$.

Recall that $S_\omega(r)\subset C^\infty(M)$ is the space spanned by eigenfunctions with spectral parameter in $[r-\omega(r), r+\omega(r)]$.  
\begin{Lemma}\label{asymp to op}
Let $\phi\in S_\omega(r)$ be a unit vector, with $\omega(r)\leq 1$, and set 
$$\Phi := \frac{1}{\sqrt{2\lfloor r^{1/2}\rfloor+1}}\sum_{|n|\leq\sqrt{r}} \phi_{2n}$$  
Then for any $K$-finite $f\in \sum_{n=-N_0}^{N_0}A_{2n}$, we have
$$I_\phi(f) = \langle f\Phi, \Phi \rangle +O_{f}(r^{-1/2})$$
\end{Lemma}

{\em Proof:}  First, we wish to show that
$$\langle f \phi_{2n}, \phi_{2m}\rangle = \langle f \phi_{2n+2}, \phi_{2m+2}\rangle +O_f(r^{-1}) +O_f(\omega(r)r^{-1})$$
for all $-\sqrt{r}\leq n,m \leq \sqrt{r}$, satisfying $|n-m|\leq N_0$ (if the latter condition is not met, both inner products are trivial, by orthogonality of the weight spaces).  We will work individually with each pair of spectral components of $\phi$, and then re-average over the spectral decomposition; therefore, we write $\phi^{(r_1)}$ and $\phi^{(r_2)}$ for the projections of $\phi$ to the eigenspaces of parameters $r_1$ and $r_2$, respectively.  Recall that $r_1, r_2 = r+O(1)$ by the quasimode condition.

We have 
\begin{eqnarray*}
\lefteqn{\langle f \phi^{(r_1)}_{2n}, \phi^{(r_2)}_{2m}\rangle }\\
& = & \frac{1}{(ir_1-n-\frac{1}{2})(-ir_2 -m-\frac{1}{2})}\langle f E^-\phi^{(r_1)}_{2n+2}, E^-\phi^{(r_2)}_{2m+2}\rangle\\
& = & \frac{\langle E^-(f\phi^{(r_1)}_{2n+2}), E^-\phi_{2m+2}^{(r_2)}\rangle - \langle E^-(f)\phi^{(r_1)}_{2n+2}, E^-\phi_{2m+2}^{(r_2)}\rangle}{(ir_1-n-\frac{1}{2})(-ir_2 -m-\frac{1}{2})}\\
& = & - \frac{\langle f\phi^{(r_1)}_{2n+2}, E^+E^- \phi^{(r_2)}_{2m+2}\rangle}{(ir_1-n-\frac{1}{2})(-ir_2 -m-\frac{1}{2})}  - \frac{\langle E^-(f)\phi^{(r_1)}_{2n+2}, \phi^{(r_2)}_{2m}   \rangle}{ir_1-n-\frac{1}{2}}\\
& = & - \frac{-ir_2+m+\frac{1}{2}}{ir_1-n-\frac{1}{2}}
\langle f\phi^{(r_1)}_{2n+2}, \phi^{(r_2)}_{2m+2}\rangle - \frac{\langle E^-(f)\phi^{(r_1)}_{2n+2}, \phi^{(r_2)}_{2m}   \rangle}{ir_1-n-\frac{1}{2}}\\
& = & \langle f\phi_{2n+2}^{(r_1)}, \phi_{2m+2}^{(r_2)}\rangle + \frac{(c_1 + c_2) \langle f\phi_{2n+2}^{(r_1)}, \phi_{2m+2}^{(r_2)}\rangle -  \langle E^-(f)\phi^{(r_1)}_{2n+2}, \phi^{(r_2)}_{2m}   \rangle}{ir_1-n-\frac{1}{2}}
\end{eqnarray*}
say, where 
\begin{eqnarray*}
c_1  =  c_1(r_1, n,m) & := & i(r -r_1) + (n-m) = O_f(1)\\
c_2 =  c_2(r_2) & := & i(r_2-r) = O_f(1)
\end{eqnarray*}
Note that $c_1$ is independent of $r_2$, and $c_2$ is independent of $r_1$.

We now average over $r_1$ to get
\begin{eqnarray*}
\lefteqn{\langle f \phi_{2n}, \phi_{2m}^{(r_2)}\rangle -  \langle f\phi_{2n+2}, \phi_{2m+2}^{(r_2)}\rangle}\\
& = &
 \left\langle f \sum_{r_1} \frac{c_1+c_2 }{ir_1-n-\frac{1}{2}} \phi_{2n+2}^{(r_1)}, \phi_{2m+2}^{(r_2)}\right\rangle - \left\langle E^-(f) \sum_{r_1}\frac{\phi_{2n+2}^{(r_1)}}{ir_1-n-\frac{1}{2}}, \phi_{2m}^{(r_2)}\right\rangle
 \end{eqnarray*}
where
\begin{eqnarray*}
\left\|	 \sum_{r_1} \frac{c_1}{ir_1-n-\frac{1}{2}} \phi_{2n+2}^{(r_1)}	\right\| & \leq & 
\left(\sum_{r_1} \left(\frac{c_1}{ir_1-n-\frac{1}{2}}\right)^2 ||\phi_{2n+2}^{(r_1)}||_2^2\right)^{1/2}  \\
& \lesssim_f  & r^{-1} \left(\sum_{r_1} ||\phi_{2n+2}^{(r_1)}||_2^2\right)^{1/2} \\
& \lesssim & r^{-1} ||\phi_{2n+2}||_2 \lesssim r^{-1}
\end{eqnarray*}
since $c_1=O_f(1)$ and $\frac{1}{ir_1-n-\frac{1}{2}} = O(r^{-1})$, and using the orthogonality of the $\phi^{(r_2)}$.  Similarly $\left\| \sum_{r_1} \frac{\phi_{2n+2}^{(r_1)}}{ir_1-n-\frac{1}{2}}\right\| = O_f(r^{-1})$ in the rightmost-term.

Thus, further averaging over $r_2$ and applying Cauchy-Schwarz gives 
\begin{eqnarray*}
\lefteqn{\langle f\phi_{2n}, \phi_{2m}\rangle - \langle f\phi_{2n+2}, \phi_{2m+2}\rangle} \\
& = &  \sum_{r_2} \left\langle f \sum_{r_1} \frac{c_1+c_2 }{ir_1-n-\frac{1}{2}} \phi_{2n+2}^{(r_1)}, \phi_{2m+2}^{(r_2)}\right\rangle - \left\langle E^-(f) \sum_{r_1}\frac{\phi_{2n+2}^{(r_1)}}{ir_1-n-\frac{1}{2}}, \phi_{2m}\right\rangle\\
& \leq & O_f(r^{-1})+ \sum_{r_2} \left\langle f \sum_{r_1} \frac{c_2}{ir_1-n-\frac{1}{2}} \phi_{2n+2}^{(r_1)}, \phi_{2m+2}^{(r_2)}\right\rangle + O_f(r^{-1})\\
& \leq & \left\langle f\sum_{r_1} \frac{1}{ir_1-n-\frac{1}{2}}\phi_{2n+2}^{(r_1)}, \sum_{r_2} c_2\phi_{2m+2}^{(r_2)}\right\rangle + O_f(r^{-1})\\
& \leq & ||f||_\infty \cdot\left\| \sum_{r_1} \frac{1}{ir_1-n-\frac{1}{2}}\phi_{2n+2}^{(r_1)}\right\| \cdot \left\| \sum_{r_2} c_2\phi_{2m+2}^{(r_2)}\right\| + O_f(r^{-1})\\
& \lesssim_f & r^{-1}
\end{eqnarray*}
using again the orthogonality of the $\phi^{(r)}$, and the fact that $c_2=O_f(1)$.  
Therefore we finally arrive at
\begin{eqnarray*}
\langle f\phi_{2n}, \phi_{2m}\rangle 
& = &  \langle f\phi_{2n+2}, \phi_{2m+2}\rangle + O_{f}(r^{-1})
\end{eqnarray*}

We iterate this $|m|\leq \sqrt{r}$ times, arriving at
$$\langle f\phi_{2n}, \phi_{2m}\rangle = \langle f\phi_{2(n-m)}, \phi_0\rangle 
+ O_{f}(\sqrt{r}r^{-1})$$
Now, by definition
$$\langle f\Phi, \Phi\rangle = \frac{1}{2\lfloor r^{1/2} \rfloor +1}\sum_{|m|,|n|\leq\sqrt{r}} \langle f\phi_{2n},\phi_{2m}\rangle$$
and so we now combine all weight space components together to get
$$\langle f \Phi, \Phi\rangle = \sum_{n=-N_0}^{N_0} \frac{(2\lfloor r^{1/2} \rfloor-2|n|+1)}{(2\lfloor r^{1/2} \rfloor+1)}\langle f\phi_{2n},\phi_0\rangle + O_{f}(\sqrt{r}r^{-1})$$
Since for all $|n|\leq N_0$ 
$$\frac{2\lfloor r^{1/2} \rfloor-2|n|+1}{2\lfloor r^{1/2} \rfloor+1} = 1 + O_f(r^{-1/2})$$
and using the $K$-finiteness of $f$, we see that
$$\langle f\Phi, \Phi \rangle = \sum_n\langle f \phi_{2n}, \phi_0\rangle + O_f(r^{-1/2}) + O_{f}(\sqrt{r}r^{-1}) $$
as required.  $\Box$

For any given sequence $\{\phi_j\}$ of quasimodes, we have constructed a sequence $\{\Phi_{j}\}_{j=1}^\infty$ such that the microlocal lifts $|\Phi_{j}|^2d\vol$ are positive measures, asymptotically equivalent to the distributions $I_{\phi_j}$.  Moreover, since $\Phi_{j}$ are constructed from $\phi_j$ via left-invariant  operators, and $T_p$ acts by isometries on the left, this construction is equivariant with respect to the action of $T_p$, and $\Phi_{j}$ will be a $T_p$-eigenfunction (resp. $\omega_j$-quasimode for $T_p$) whenever $\phi_j$ is.

\subsection{Asymptotic Invariance under the Geodesic Flow}
We now turn to the invariance under the geodesic flow.  By Lemma~\ref{asymp to op}, we may consider the distribution 
$$f\mapsto \langle f \Phi_\infty, \phi\rangle$$
in place of our positive-measure microlocal lift, since it is more amenable to verifying this invariance.

\begin{Lemma}\label{geo flow invariance}
Let $\{\phi_j\}$ be a sequence of $\omega(r_j)$ quasimodes with approximate parameter $r_j$, satisfying $r_j\to\infty$ and $\omega(r_j)\to 0$.  Then for any $f\in C^\infty(S^*M)$, we have
$I_{\phi_j}(Hf) \to 0$ as $r_j\to \infty$.
\end{Lemma}

{\em Remark:}  In contrast with Lemma~\ref{asymp to op}, here it is necessary to assume that $\omega(r_j) \to 0$; in fact, the sequences of quasimodes considered in \cite{localized_example} have $\omega(r_j)\asymp 1$, and do not satisfy  Lemma~\ref{geo flow invariance}.

{\em Proof:}
Once again, we may assume that $f\in\sum_{n=-N_0}^{N_0} A_{2n}$ is $K$-finite, and it will be natural to consider the contribution of each pair of spectral parameters individually; so we again write $\phi^{(r_1)}$ and $\phi^{(r_2)}$ for the projections of $\phi$ to the eigenspaces of parameters $r_1$ and $r_2$ respectively, and similarly $\Phi_\infty^{(r_1)}$.  

Recall that we have
$$-(r_2^2+\frac{1}{4})\langle f \Phi_\infty^{(r_1)}, \phi^{(r_2)}\rangle = \langle f\Phi_\infty^{(r_1)}, E^-E^+ \phi^{(r_2)}\rangle = \langle E^-E^+ (f\Phi_\infty^{(r_1)}), \phi^{(r_2)}\rangle$$
and use the identities (\ref{identities}) to compute 
\begin{eqnarray*}
\lefteqn{E^-E^+(f\cdot\Phi_\infty)}\\ 
& = & f\cdot(E^-E^+\Phi_\infty) + (E^+f)\cdot(E^-\Phi_\infty) + (E^-f)\cdot(E^+\Phi_\infty) + (E^-E^+f)\cdot\Phi_\infty\\
& = & -(r_1^2+\frac{1}{4})f\cdot \Phi_\infty + f\cdot(D_1(W)\Phi_\infty) + 2ir_1(Hf)\cdot\Phi_\infty \\
&  &+  (E^+f)\cdot(-\frac{1}{2}+\frac{i}{2}W)\Phi_\infty + (E^-f)\cdot(-\frac{1}{2}-\frac{i}{2}W)\Phi_\infty + (E^-E^+f)\cdot\Phi_\infty
\end{eqnarray*}
where $D_1(W)$ is a differential operator in $W$.  This means, after integrations by parts in $W$ (recalling that $W\phi^{(r_2)}=0$) we get
\begin{eqnarray*}
\lefteqn{-(r_2^2+\frac{1}{4})\langle f \Phi_\infty^{(r_1)}, \phi^{(r_2)}\rangle}\\
&  = &  -(r_1^2 + \frac{1}{4})\langle f\Phi_\infty^{(r_1)}, \phi^{(r_2)}\rangle + 2ir_1\langle Hf \Phi_\infty^{(r_1)}, \phi^{(r_2)}\rangle + \langle (D_2f) \Phi_\infty^{(r_1)}, \phi^{(r_2)}\rangle
\end{eqnarray*}
for a fixed differential operator $D_2$.  But since $r_1,r_2 = r + o(1)$, we have $r_1^2 - r_2^2 = o(r)$, and so putting all of the spectral components back together again we see
\begin{eqnarray*}
(r_1^2-r_2^2) \langle f\Phi_\infty^{(r_1)}, \phi^{(r_2)}\rangle & = & 2ir_1  \langle (Hf) \Phi_\infty^{(r_1)}, \phi^{(r_2)}\rangle + \langle (D_2f) \Phi_\infty^{(r_1)}, \phi^{(r_2)}\rangle\\
(r_1^2-r_2^2) \langle f\Phi_{N_0}^{(r_1)}, \phi^{(r_2)}\rangle & = & 2ir_1  \langle (Hf) \Phi_{\infty}^{(r_1)}, \phi^{(r_2)}\rangle + \langle (D_2f) \Phi_\infty^{(r_1)}, \phi^{(r_2)}\rangle\\
o(r) ||f||_\infty  ||\Phi_{N_0}||_2  ||\phi||_2  & = & 2ir\langle (Hf) \Phi_\infty, \phi \rangle + o_f(1)||\Phi_{N_0+1}||_2  + I_\phi (D_2f)\\
o_f(r) & = & 2ir I_\phi(Hf) +o_f(1) + O_f(1)
\end{eqnarray*}
where we may replace $\Phi_\infty$ on the left with $\Phi_{N_0}$, the projection of $\Phi_\infty$ to $\sum_{|n|\leq N_0}A_n$, by orthogonality of weight spaces--- and similarly on the right side since $f\in \sum_{|n|\leq N_0}A_n$ implies that $Hf \in \sum_{|n|\leq N_0+1}A_n$---  noting that $||\Phi_{N_0}||_2 = \sqrt{2N_0+1} = O_f(1)$ and similarly $||\Phi_{N_0+1}||_2=O_f(1)$.

The result now follows by dividing by $r$ and taking the limit as $j\to\infty$.  $\Box$

\subsection {Microlocal lift on $\Gamma \backslash \H \times \H$}\label{HxH lift}

Given a quasimode $\phi$ on $\Gamma \backslash \H$, we constructed an element $\Phi \in L ^2 \left (\Gamma \backslash \PSL (2, \R) \right)$ satisfying ($\Phi $1)--($\Phi$3) on p.~\pageref{Greek Phi properties} by considering $\phi$ as a right $K=\SO (2)$-invariant function on $\Gamma \backslash \PSL (2, \R)$ and applying a carefully chosen generalized differential operator of the type $D = \sum_ {i = 1}^k f_i(\Omega) X _ i$ with $f_i$ appropriately chosen analytic functions, $\Omega$ the Casimir operator and $X _ i$ some invariant differential operators on $\Gamma \backslash \PSL (2, \R)$. We note that the choice of $D$ depends on the approximate eigenvalue of $\phi$.
Since $\Omega$ is in the center of the algebra of invariant differential operators it is easy to make sense very concretely of the operator $D$ by decomposing $L ^2 (\Gamma \backslash \PSL (2, \R))$ to $\omega$-eigenspaces and applying --- on the eigenspace corresponding to eigenvalue $r$ ---the honest differential operator $\sum_ {i = 1} ^ k f _ i (r) X _ i$.

Now suppose that we are working not on $\Gamma \backslash \H$, but on $X = \Gamma \backslash \H \times \H$, with $\Gamma$ an irreducible lattice in $\PSL (2, \R) \times \PSL (2, \R)$. Suppose $\Delta _ 1$ is the Laplacian operator acting on the first $\H$ component, and $\Delta _ 2$ the Laplacian on the second such component. If $\phi$ is an approximate eigenfunction of both $\Delta _ 1$ and $\Delta _ 2$ with approximate eigenvalues $r _ 1, r _ 2$ with $r _ 1$ large, we could take the same generalized differential operator $D$ discussed in the previous paragraph and consider it as an operator on $\Gamma \backslash \PSL (2, \R) \times \H$. Then if $\phi$ is considered as a $\SO (2) $-invariant function on $\Gamma \backslash \PSL (2, \R) \times \H$, taking $\Phi = D \phi$ we obtain a function satisfying the analogous conditions to ($\Phi$1)--($ \Phi$3) for $\Gamma \backslash \PSL (2, \R) \times \H$ as explained in greater detail in \cite{LinHxH}.

\section{The Propagation Lemma}\label{propagator}

Our methods were inspired in part by the work of Anantharaman et. al. (eg., \cite{Anan, AN, AKN, Riv1, Riv2}); in particular, the observation that one can obtain interesting information on eigenfunctions--- and quasimodes--- by breaking the function up into smaller pieces, letting each piece disperse individually through the quantum dynamics, and adding the pieces back together.  In this spirit, our approach is to use an analogous ``wave propagation" on the Hecke tree to disperse pieces of our $T_p$-quasimodes, as described in Lemma~\ref{propagation}.  As will become clear later on, this dispersion alone is not enough to get Theorem~\ref{main}; we will need to employ interferences in order to build a better dispersion mechanism, that amplifies a desired spectral window.

\subsection{Constructing the Kernel}\label{K_N construction}

The following lemma proved along the lines of \cite{meElon} is central to our approach:
\begin{Lemma}\label{operator}
Let $0<\eta< 1/2$.
For any sufficiently large $N \in \mathbb{N}$ (depending on $\eta$), and any $\theta_0 \in [0, \pi]$,
there exists an operator $K_N$ on $S^*M$ satisfying:
\begin{enumerate}[leftmargin=*,label=\textup{(\arabic*)}]
\item\label{support of operator}
$K_N(\delta_x)$ is supported on the union of Hecke points $y \in T_{p^j}(x)$ up to distance $j \leq N$ in the Hecke tree.
\item\label{small matrix coefficients}
$K_N$ has matrix coefficients bounded by $O(p^{-N \delta})$, in the sense that for any $x \in S^*M$
$$|K_N(f)(x)|\lesssim p^{-N \delta} \sum_{j=0}^{N} \sum_{y \in S_{p^j}(x)} |f(y)|$$
where $\delta$ depends only on $\eta$ (explicitly, can be taken to be $\eta ^2 / 512$).
\item\label{weak positivity property} Any $T_p$ eigenfunction is also an eigenfunction of $K_N$ with eigenvalue~$\geq -1$.  
\item Eigenfunctions with $T _ p$-eigenvalue $2 \cos \theta$ with $\absolute {\theta - \theta _ 0} \leq \frac {1 }{ 2 N}$, as well as all untempered eigenfunctions ($T _ p$-eigenvalue $\not \in [-2,2]$) have $K_N$-eigenvalue~$>\eta^{-1}$.
\end{enumerate}
\end{Lemma}

Lemma~\ref{operator} is based on the well known connection between Hecke operators and Chebyshev polynomials. A way to derive these which we have found appealing is via the following $p$-adic wave equation for, say, compactly supported functions on~$\mathcal{T}_{p+1}$ :
\begin{eqnarray*}
\Phi_{n+1} & = & \frac{1}{2}T_p\Phi_n - \left(1-\frac{T_p^2}{4}\right)\Psi_n\\
\Psi_{n+1} & = & \frac{1}{2}T_p\Psi_n + \Phi_n
\end{eqnarray*}
which is a discrete analog of the non-Euclidean wave equation (more precisely, of the unit time propagation map for the wave equation) on~$\mathbb{H}$.  For initial data $(\Phi_0, \Psi_0)$, the solution to this equation is given by the sequence
\begin{eqnarray*}
\Phi_n & = & P_n\left[\frac{1}{2}T_p\right] \Phi_0 - \left(1-\frac{T_p^2}{4}\right) Q_{n-1}\left[\frac{1}{2}T_p\right]\Psi_0\\
\Psi_n & = & P_n\left[\frac{1}{2}T_p\right] \Psi_0 + Q_{n-1}\left[\frac{1}{2}T_p\right]\Phi_0
\end{eqnarray*}
where $P$ and $Q$ are Chebyshev polynomials of the first and second kinds, respectively, given by
\begin{eqnarray*}
P_n(\cos{\theta}) & = & \cos{n\theta}\\
Q_{n-1}(\cos{\theta}) & = & \frac{\sin{n\theta}}{\sin{\theta}}
\end{eqnarray*}
and preserves the energy-like quantity
\begin{equation*}
\norm {\Phi} ^2 + \left \langle \Psi, \left (1- \frac {T _ p ^2} {4} \right) \Psi \right \rangle.
\end{equation*}
This can be proved directly by induction, using the well-known recursive properties of the Chebyshev polynomials:
\begin{eqnarray*}
P_{n+1}(x) & = & xP_n(x) - (1-x^2) Q_{n-1}(x)\\
Q_n(x) & = & xQ_{n-1}(x) + P_n(x)
\end{eqnarray*}

Suppose we take initial data $(\delta_0, 0)$.  The solution to the $p$-adic wave equation is then $\{(P_n[\frac{1}{2}T_p]\delta_0, Q_{n-1}[\frac{1}{2}T_p]\delta_0)\}$.  On the other hand, one can compute the explicit solution inductively\footnote{An alternative proof, with a more spectral flavor, was given in \cite{meElon}.}; looking at the first coordinate, we get the following ``Propagation Lemma" on the tree:
\begin{Lemma}\label{propagation}
Let $\delta_0$ be the delta function at $0$ in the $p+1$-regular tree  $\mathcal{T}_{p+1}$.  Then for $n$ even, we have
\begin{eqnarray*}
P_n\left[\frac{1}{2}T_p\right]\delta_0(x) & = & \left\{ \begin{array}{ccc} 0  & \quad & |x|  \text{ odd } \quad \text{or} \quad |x|>n\\ \frac{1-p}{2p^{n/2}} & \quad & |x|<n \quad \text{and} \quad |x| \text{ even } \\ \frac{1}{2p^{n/2}} & \quad & |x| = n \end{array}\right.
\end{eqnarray*}
In particular, we have
$$P_n\left[\frac{1}{2}T_p\right]\delta_0(x) \lesssim p^{-n/2} \qquad \text{for all $x\in \mathcal{T}_{p+1}$}.$$
\end{Lemma}

We now have a description of the $p$-adic wave propagation in both spectral and spacial terms, which we will use to construct our desired radial kernel $K_N$ on the Hecke tree.  It will be convenient to parametrize the $T_p$-eigenvalues as $2\cos(\theta)$, where
\begin{itemize}
\item{The tempered spectrum is parametrized by $\theta\in [0, \pi]$.}
\item{The positive part of the untempered spectrum has $i\theta \in (0,\log{\sqrt{p}})$. }
\item{The negative part of the untempered spectrum has $i\theta+\pi \in (0,\log{\sqrt{p}})$. }
\end{itemize}
We also recall that for any (say, compactly supported) radial kernel $k$ on $\mathcal{T}_{p+1}$, any $T_p$-eigenfunction $\phi$ is also an eigenfunction of convolution with $k$, with eigenvalue--- depending only on the $T_p$-eigenvalue of $\phi$--- given by the {\bf spherical transform} $h_k(\theta)$.

Now, Lemma~\ref{propagation} is a good start towards Lemma~\ref{operator}, since the kernel $P_n[\frac{1}{2}T_p]\delta_0$ satisfies the first two properties of Lemma~\ref{operator}, with $\delta=1/2$.  Unfortunately, the spherical transform of this kernel is $2\cos{n\theta}$, which is bounded by $2$ and cannot satisfy the third condition of Lemma~\ref{operator}.  More crucially, it takes negative values that are as large in absolute value as the maximum.  In order to create the kernel we want, we will have to combine different $P_n[\frac{1}{2}T_p]\delta_0$ waves together, in such a way that they interfere constructively at our chosen spectral interval, to generate a large eigenvalue there, but do not become too negative elsewhere.  To allow this kind of interference, however, we will have to sacrifice somewhat in our bounds for $\delta$.

\begin{proof}[Proof of Lemma~\ref{operator}]
The case where $\theta_0=0$ is a bit simpler, so we consider this case first. Using the Hecke correspondence, we can assign to any ``kernel'' (spherically invariant compactly supported function) on a $p + 1$-regular tree with a marked point an operator on $L ^2 (S^*M)$. We will produce the operator $K _ N$ satisfying the properties of Lemma~\ref{operator} from such a kernel $k _ N$, which in turn will be defined from its spherical transform $h_{k_N}$.
Denoting the Fej\'er kernel of order $L$ by $F_L(\theta)=\frac{1}{L} \left(\frac{\sin (L \theta/2)}{\sin(\theta/2)}\right)^2$, we will set the spherical transform to be
$$h_{k_N}(\theta) = F_L(q \theta)-1$$
for an appropriately chosen $q$.  Now $F_L \gtrsim L$ on $(-\frac{1}{L},\frac{1}{L})$, which contains the spectral parameters for components of $\Phi_j$ if $N \geq L$, and $F_L$ is non-negative, so the third condition of Lemma~\ref{operator} is satisfied on the tempered spectrum, as long as $N>L>\eta^{-1}+1$.  Moreover, we can write
$$F_L(q \theta)-1 = \sum_{j=1}^L \frac{2(L-j)}{L} \cos{(jq \theta)}$$
and observing that $\cos(jq \theta)>\cos(0)$ on the entire untempered spectrum as long as $q$ is even, we see that the third condition holds on the full spectrum.
We also observe that
\begin{eqnarray*}
|k_N(x)| & \leq & \sum_{j=1}^L 2\left|P_{jq}\left[\frac{1}{2}T_p\right]\delta_0 \right| \\
& \lesssim & \sum_{j=1}^L p^{-jq/2}\\
& \lesssim & p^{-q/2}
\end{eqnarray*}
which provides the second condition of Lemma~\ref{operator}, as long as $q \geq 2N \delta$.  Moreover, since each $P_{jq}[\frac{1}{2}T_p]\delta_0$ vanishes outside the ball of radius $Lq$, the first condition is satisfied whenever $Lq \leq N$.  So we may take $L=\lceil \eta^{-1}\rceil+1$, and $q=2 \lfloor N/2L \rfloor$, which yields $\delta = \lfloor q/2N \rfloor \gtrsim \eta$.  The same kernel also works for $\theta=\pi$ and the untempered spectrum.

We must now  consider the case of $\theta_0 \in (0,\pi)$.  The general idea is the same, but in order to maintain positivity on the untempered spectrum, we can't simply shift the Fej\'er kernel to a different approximate eigenvalue.  The solution to this problem is to find a suitable multiple of $\theta_0$, which is sufficiently close to $0$, and  choose our value of $q$ to be divisible by this multiple.  Thus the original Fej\'er kernel evaluated at $q \theta_0$ will still be large, as before, without affecting its positivity on the untempered spectrum.
To guarantee that we can find such a suitable value of $q$, we will invoke Dirichlet's Theorem on Diophantine approximation.  
 We will lose some in our bounds for $\delta$ (as well as in the implied constant)--- whereas our previous argument gave $\delta \gtrsim \eta$, we will have to settle for a more modest $\delta \gtrsim \eta^2$ in order to achieve the necessary flexibility in choosing $q$ .  The following argument is essentially identical to the one appearing
in \cite{meElon}, with some tweaking to the constants to accommodate our quasimodes.

Set $L=\lfloor \eta^{-1}\rfloor$ and $Q=\lceil \frac{1}{8}N \eta \rceil$.  By Dirichlet's Theorem, we can find a positive integer $q \leq Q$ such that $|q \theta_0 \bmod{2 \pi}| < 2 \pi Q^{-1}$.  There exists an even multiple of $q$, say $q' = 2lq$, such that $\frac{1}{64}Q \eta \leq q' \leq 2Q$ --- indeed, if $q \geq \frac{1}{128}Q \eta$, we can simply take $l=1$; otherwise there is a multiple of $q$ between $\frac{1}{128}Q \eta$ and $\frac{1}{64}Q \eta$, so take twice that multiple. Either way, since $N$ is assumed to be large depending on $\eta$, we may assume $Q \eta>64$ which implies that both when $l=1$ and $l>1$ we have that
\begin{equation*}
 2l  < \frac{1}{32}Q \eta \qquad \text{and} \qquad |q'\theta_0 \bmod{2 \pi}| < \frac{1}{16}\pi \eta
\end{equation*}
We now set the spherical transform of $k_N$ to be $h_{k_N}(\theta) = F_{2L}(q'\theta) - 1$, where $F_{2L}$ is the Fej\'er kernel of order $2L$, and $K _ N$ the corresponding operator on $L ^2 (S ^{*} M)$ i.e.
 \begin{equation*}
 K _ N = \sum_ {j = 1} ^ {2 L} \frac {2 L - j }{ L} P _ {j q '} \left (\frac {T _ p }{ 2} \right)
 .\end{equation*}

For any $T _ p$-eigenfunction with eigenvalue $2 \cos\theta$ with $\theta \in [\theta_0-\frac{1}{2}N^{-1}, \theta_0+\frac{1}{2}N^{-1}]$,   we have that \[
|q'(\theta_0 - \theta)|\leq QN^{-1} \leq \eta/8 + 1/N < \eta/6;
\]
 this means that
$$|q'\theta \bmod{2 \pi}| < \frac{1}{16}\pi \eta + \frac{1}{6}\eta < \frac{1}{8}\pi \eta \leq \frac{\pi}{8L}.$$
It follows that
\begin{eqnarray*}
F_{2L}(q'\theta) =  \frac{1}{2L} \frac{\sin^{2}(Lq'\theta)}{\sin^2(q'\theta/2)}
& \geq & \frac{2}{L}\frac{\sin^{2}(Lq'\theta)}{(q'\theta)^2}\\
& \geq & 2L \left(\frac{\sin(Lq'\theta)}{Lq'\theta}\right)^2
\end{eqnarray*}
But since $L \in \mathbb{Z}$, and $q'\theta \bmod{2 \pi}\in \left(-\frac{\pi}{8L}, \frac{\pi}{8L}\right)$, we have $Lq'\theta \bmod{2 \pi} \in [-\frac{\pi}{8}, \frac{\pi}{8}]$, which implies that
$$\left|\frac{\sin(Lq'\theta)}{Lq'\theta}\right| \geq \frac{\sin{\frac{\pi}{4}}}{\frac{\pi}{4}} = \frac{1}{\sqrt{2}}\frac{4}{\pi}$$
whereby
\begin{eqnarray*}
F_{2L}(q'\theta) & \geq & 2L \left(\frac{1}{\sqrt{2}}\frac{4}{\pi}\right)^2\\
& \geq & 2L \frac{8}{\pi^2} > L+1
\end{eqnarray*}
as long as $L \geq 2$ (which follows from the hypothesis $\eta<1/2$).
Therefore the $K _ N$-eigenvalue of this eigenfunction will be $>L+1 \geq \eta^{-1}$.  Moreover, since $F_{2L}$ is positive, the spherical transform of $k_N$ is bounded below by $-1$; note also that on the untempered spectrum of $T _ p$, then $K _ N$-eigenvalue is also $>L+1$ .  It remains to check the first two properties.

Now, by Lemma~\ref{propagation}, we see that the kernel whose spherical transform is $\cos{2j \theta}$--- i.e., the kernel of $P_{2j}(\frac{1}{2}T_p)$--- has sup-norm $\lesssim_p p^{- j}$.  The spherical transform of $k_N$ is a sum of terms of the form $\frac{2L-j}{L}\cos{jq'\theta}$, where $j=1,2 ,\ldots, 2L$ (note that we eliminated the $j=0$ term by subtracting off the constant contribution to $F_{2L}$) and $q'\in 2 \mathbb{Z}$.  Thus
$$ ||k_N||_\infty \lesssim  \sum_{j=1}^L p^{-jq'}
\lesssim_{p} p^{-q'}$$
Then, since
$$q'\geq \frac{1}{64}Q \eta \geq \frac{1}{512}N \eta^2$$
the second condition is satisfied with $\delta = \frac{1}{512} \eta^2$.
Moreover, since each kernel $P_{jq'}(\frac{1}{2}T_p)$ is supported in a ball of radius $jq'\leq 2L \cdot 2Q < N$, the full kernel $k_N$ is supported in a ball of radius $N$.
This concludes the proof of Lemma~\ref{operator}.
\end{proof}

The kernel $k_N$ produced by Lemma~\ref{operator} will be used in section~\ref{Hecke} to establish positive entropy on a.e. ergodic component of quantum limits arising from joint $o(1)$-quasimodes.
It will also be used in section \ref{Hecke recurrence subsection}, though in a less delicate way (in particular, without making use of property \ref{weak positivity property} of Lemma~\ref{operator}).

\subsection{Hecke Recurrence for Quasimodes}\label{Hecke recurrence subsection}

As shown in \cite {Lin} (and implicitly already in \cite{LinHxH}) a quantum limit arising from a sequence of $T _ p$-eigenfunctions is Hecke recurrent. This remains true for $T _ p$-quasi-modes, and in order to streamline the presentation we shall make use of the discussion in section \ref{K_N construction}. 
The recurrence property for quantum limits arising from our joint $o(1)$ quasimodes, follows immediately from the following estimate (see \cite{Lin} for details):

\begin{Lemma}\label{Hecke recurrence by propagation lemma}
Let $\{ \Phi_j \}$ be a sequence of $\omega_j$-quasimodes for $T_p$, with $\omega_j \to 0$, such that the sequence $|\Phi_j|^2 d\vol$ converges weak-* to a measure $\mu$.   Let $x \in \supp \mu \subset S^*M$, and $B$ a small open ball around $e \in G$ (say, of radius less than $1/3$ the injectivity radius of $S^*M$).  Then
$$ \liminf_{j \to \infty} \frac{\sum_{d _ p (x,y)\leq N} \int_{y\overline{B}} |\Phi_j|^2}{\int_{xB} |\Phi_j|^2} \to \infty \qquad \text{as $N \to \infty$} $$
uniformly in $x$ and the radius of $B$.
\end{Lemma}
Here, $d _ p (x,y)$ refers to distance in the Hecke tree containing $x$ and $y$, i.e. the smallest $d$ for which $x \in T _ {p ^ d} (y)$; in particular, the sum is finite.
While we have in mind the case where the $\Phi_j$ arise from a sequence of joint $\Delta$ and $T _ p$ quasimodes on $M$ as in Lemma~\ref{asymp to op}, the proof of Lemma~\ref{Hecke recurrence by propagation lemma} only uses the $T_p$-structure, and holds for any sequence of $o(1)$-quasimodes for $T_p$ on $S^*M$.

\begin{proof}[Proof of Lemma~\ref{Hecke recurrence by propagation lemma}]
We will use a simplified version of the construction in Lemma~\ref{operator}. Let $L$ be a large integer. Suppose $\Phi _ j$ is a $\omega _ j$-quasimode for $T _ p$ with approximate eigenvalue $\lambda _ j$ and set $\theta _ j \in [0, 2 \pi]$ by
\begin{equation*}
\theta _ j = \begin{cases}
\cos ^{-1} (\lambda _ j / 2)& \text{if $\lambda _ j \in [-2,2]$} \\
0& \text{if $\lambda _ j > 2$} \\
\pi& \text{if $\lambda _ j < - 2$}
\end{cases}
.\end{equation*}
Find a $q \in \left\{ 1, \dots, 100L \right\}$ so that
\begin{equation}\label{choice of q}
\absolute {q \theta _ j \bmod 2 \pi} \leq \frac \pi{50L},
\end{equation}
 and set
\begin{equation*}
K  = \sum_ {l = 1} ^ {L} P _ {2ql} \left (\frac {T _ p }{ 2} \right)
.\end{equation*}

By Lemma~\ref{propagation} and two consecutive applications of the Cauchy Schwarz inequality
\begin{align*} 
\absolute {K \Phi _ j (x)} &\lesssim  \absolute {\sum_ {l = 1} ^ { L} \sum_ {y: \, 2q(l-1) < d_p(y,x) \leq 2ql} p ^ {-2ql/2} {\Phi _ j (y)} } \\
& \leq \sum_ {l = 1} ^ { L} \left (\#\left\{ y: 2q(l-1) < d_p(y,x) \leq 2ql \right\} \cdot p ^ {-2ql}\cdot \sum_ {\text{\makebox[50pt]{$\smash{y: \, 2q(l-1) < d_p(y,x) \leq 2ql}$}}} \absolute {\Phi _ j (y)} ^2 \right) ^ {1/2} \\
& \lesssim \sum_ {l = 1} ^ { L} \left (\sum_ {\smash{y: \, 2q(l-1) < d_p(y,x) \leq 2ql}} \absolute {\Phi _ j (y)} ^2 \right) ^{1/2}  \leq 
\left (L \sum_ {\text{\smash{\makebox[58pt]{${d_p(y,x) \leq  100 L^2}$}}}} \absolute {\Phi _ j (y)} ^2 \right) ^{1/2}
\hspace{-10pt}.\end{align*}
and similarly for every $g \in B$
\begin{equation*}
\absolute {K \Phi _ j (x g)} \lesssim \left (L \sum_ {\text{\smash{\makebox[58pt]{${d_p(y,x) \leq  100 L^2}$}}}} \absolute {\Phi _ j (yg)} ^2 \right) ^{1/2}
.\end{equation*}
Set $a = \sum_ {l = 1} ^ L P _ {2 q l} (\lambda _ j/2)$ (which we can also write in the tempered case as $\sum_ {l = 1} ^ L \cos (2 q l \theta _ j)$); by considering separately the tempered and untempered cases (using in the former case \eqref{choice of q}) it can be verified that $a  \gtrsim
L$. Since $\Phi _ j$ is an $\omega _ j$-quasimode
\begin{equation*}
\norm {K \Phi _ j - a \Phi _ j} = O_{L} (\omega _ j)
.\end{equation*}
It follows that
\begin{align*}
a ^2  \int_ {x B} \absolute {\Phi _ j} ^2 & \leq \int_ {x B} \absolute {K \Phi _ j} ^2 + O _ {L} (\omega _ j) \vol (B) ^{1/2} \\
& \lesssim L \sum_{d_p(y,x) \leq  100 L^2} \int_ {y B} \absolute {\Phi _ j} ^2 + O _ {L} (\omega _ j) \vol (B) ^{1/2}
.\end{align*}
Since $a ^2 \gtrsim L ^2$, $\omega _ j \to 0$, and $x \in \supp \mu$ (so that $\liminf_ {j \to \infty} \int_ {x B} \absolute {\Phi _ j} ^2 > 0$) we conclude that
\begin{equation*}
\liminf_{j \to \infty} \frac{\sum_{d _ p (x,y)\leq 100L^2} \int_{y \overline{B}} |\Phi_j|^2}{\int_{xB} |\Phi_j|^2} \gtrsim L
.\end{equation*}
\end{proof}

\section{Proof of Theorem~\ref{main}}\label{Hecke}
Let
$$B(\epsilon, \tau) = \left\{ a(t) u ^ - (s _ -) u ^ + (s _ +): t \in (-\tau, \tau), s _ -, s _ + \in (-\epsilon, \epsilon) \right\}
$$
where
\begin{eqnarray*}
u^+(s) & = & \begin{pmatrix} 1 & 0\\ s & 1 \end{pmatrix}\\
u^-(s) & = & \begin{pmatrix} 1 & s\\ 0 & 1 \end{pmatrix}\\
a(t) & = & \begin{pmatrix} e^{t/2} & 0\\ 0 & e^{-t/2} \end{pmatrix}.
\end{eqnarray*}

We recall the setup, which is the same as in \cite{RS}: we begin with a quaternion division algebra
\begin{equation*}
H (\Q) = \left\{ x + i y + j z + i j w: x, y, z, w \in \Q \right\}
\end{equation*}
with $i^2=a, j^2=b, ij=-ji$ and the usual norm $n(x + i y + j z + i j w) = x ^2 - a y ^2 - b z ^2 + ab w ^2$ and trace $tr(x + i y + j z + i j w) = 2x$; that $H (\Q)$ is a division algebra is equivalent to $n (\alpha) \ne 0$ for all nonzero $\alpha \in H (\Q)$; we also assume $a>0$ and then
\begin{equation*}
\iota ( x + i y + j z + i j w) = \begin{pmatrix} x + \sqrt a y& z + \sqrt a w \\
b(z - \sqrt a w) & x - \sqrt a y \end{pmatrix}
\end{equation*}
gives an embedding of $H(\Q)$ to $\operatorname{Mat}_2(F)$, \ $F=\Q (\sqrt a)$, which extends to an isomorphism of rings between $H (\R)$ to $\operatorname{Mat}_2(\R)$. An order $R<H(\Q)$ is a subring containing 1 which as an additive group is of rank 4 and so that $tr(\alpha) \in \Z$ for every $\alpha \in R$. An example of an order is $\tilde R= \mathcal{O} _ F + j \mathcal{O} _ F$ (with $F$ embedded in $H(\Q)$ in the obvious way). We take $R$ to be a maximal order containing $\tilde R$; the assumptions that $R$ is maximal is not important, but makes it easier to write things accurately in classical language. Cf. \cite{Eichler} for more details.

The following estimate can be derived using the techniques of \cite{BorLin}, specifically Lemmas~3.1 and 3.3 there (much more general statements of this type by Silberman and Venkatesh can be found in \cite{Lior-Akshay2}).  We include the proof below for completeness.
\begin{Lemma}\label{disjoint}
For $\tau$ fixed but small enough, there exists a constant $c$ (depending only on $\tau$), such that for any $x,z \in X = \Gamma \backslash \PSL(2,\mathbb{R})$, and any
$\epsilon<cp^{-2N}$, the tube $zB(\epsilon, \tau)\subset X$ contains at most $O(N)$ of the Hecke points $\bigcup_ {j \leq N} T_{p^j} (x)$.
\end{Lemma}

\begin{proof}
Let $\mathcal{F}$ be a compact fundamental domain for $\Gamma$
and~$\underline x, \underline z \in \mathcal{F}$ points projecting to $x, z$ respectively.
Suppose there exist $k$ distinct Hecke points $y_i \in T_{p^{j_i}}(x)$ with $j_i \leq N$, such that
$y_i \in zB(\epsilon, \tau)$.
By construction of the Hecke correspondence, this implies that there are $\alpha _ i \in R (p ^ {j _ i})$ so that $\underline {\alpha _ i  x} \in \underline z B (\epsilon, \tau)$.
\smallskip

Then, since $B(\epsilon, \tau)^{-1}B(\epsilon, \tau)\subset B(4 \epsilon, 3 \tau)$ for $\epsilon$ sufficiently small 
we have
\begin{equation}
\label{underlined inclusion} \underline{\alpha_1}^{-1}\underline{\alpha _ i}\in \underline xB(4 \epsilon, 3 \tau) \underline x^{-1}\cap \iota\left(\bigcup_{j \leq 2N}R(p^j) \right)
\end{equation}
where $\iota$ is as above our isomorphism of the quaternion division algebra $H$ to $2 \times 2$ matrices over $\mathbb{R}$. Set, for $i = 1, \dots, k - 1$, $\beta _ i = \overline { \alpha _ 1} \alpha _ {i + 1}$; since $\overline { \alpha} = tr(\alpha)-\alpha$ is in $R$ iff $\alpha$ is, we have that $\beta _ i \in R$, $1 \leq n (\beta _ i) \leq p^{2 N}$ and of course $\underline {\beta _ i} = \underline{\alpha_1}^{-1}\underline{\alpha _ i}$.
From \eqref{underlined inclusion} it follows that
\begin{align*}
\absolute {tr(\beta _ i \beta _ j - \beta _ j \beta _ i)} & <\frac1{c_1} n(\beta _ i \beta _ j) ^ {1/2} \epsilon \\
tr(\beta _ i ) &\geq  n(\beta _ i)^{1/2}(2 - \frac\epsilon{c_1})\\
tr(\beta _ i ^2) &\geq  n(\beta _ i)(2 - \frac\epsilon{c_1}) 
\end{align*}
with $c_1$ depending only on $\tau$. As $tr (\beta _ i^2), tr(\beta _ i \beta _ j - \beta _ j \beta _ i) \in \Z$ it follows that if $\epsilon <c_1p^{-2N}$ we have that $\beta _ i \beta _ j = \beta _ j \beta _ i$ and $tr (\beta _ i^2) \geq 2n(\beta _ i)$. It also follows that $tr (\beta _ i) > 0$, hence since $tr (\beta _ i ^2) = tr (\beta _ i) ^2 - 2 n(\beta _ i)$ we find that $tr(\beta _ i) \geq 2 n (\beta _ i) ^ {1/2}$.

By construction $\beta _ i \not\in \Q$, so the only elements in $H (\Q)$ that commute with $\beta _ 1$ are $\Q (\beta _ 1)$. Since $tr(\beta _ i) \geq 2 n (\beta _ i) ^ {1/2}$, the field $\Q (\beta _ 1)$ is isomorphic to a real quadratic number field $L$. Let $\phi:L \to \Q (\beta _ 1) \subset H (\Q)$ be this isomorphism.

Since $R$ is an order, for any $\beta \in \Q (\beta _ 1) \cap R$ (such as $\beta _ i$ for $1 \leq i \leq k - 1$)  we have that $\phi ^{-1} (\beta)$ is in $\mathcal{O} _ L$, the ring of integers in $L$. Hence for any $1 \leq i \leq k - 1$ we obtain a principal ideal $I_i = \phi ^{-1} (\beta _ i) \mathcal{O} _ L$ of norm $n(\beta _ i) \leq p^{2N}$; write $I'_i=p^kI_i$ with $I_i$ not divisible by $p$.

Since $\beta _ i$ are all distinct,
if $I'_i=I'_{j}$ then $\beta _ i = \beta _ j p^k \theta$ with $\theta \in \mathcal{O} _ K ^ {*}$ and $k \in \Z$. Since $\underline {\beta _ i}, \underline {\beta _ j} \in \underline{x}B(4 \epsilon, 3 \tau) \underline x^{-1}$, this would imply that $\underline \theta \in \underline{x}B(c \epsilon, c \tau) \underline x^{-1}$ for an appropriate constant $c$.  But then if $\tau$ was chosen sufficiently small, we must have $\theta =1$ and $\underline {\beta _ i}= \underline {\beta _ j} $, in contradiction to $\underline {\alpha _ j}$ being all distinct.

Thus the map $j \mapsto I ' _ j$ is injective, and since in $\mathcal{O} _ L$ there are at most $4N$ ideals of norm dividing $p^{2N}$  which are not divisible by the principal ideal $p \mathcal{O} _ L$, we conclude that $k \leq 4N$.
\end{proof}

Modifying the constant $c$, one can get the seemingly stronger conclusion that given any $x \in X$, for at most $O (N)$ of the points $y \in \bigcup_ {j \leq N} S _ {p ^ j} (z)$ the intersection $xB(cp^{-2N}, \tau) \cap y B (c p ^ {- 2 N}, \tau) \neq \emptyset$.

\begin{proof}[Proof of Theorem~\ref{main}]
Take $\mathcal{P}$ to be a partition  of $S^*M$ with the property that $\mu (\partial P) = 0$ for every $P \in \mathcal{P}$, and such that $\max_ {P \in \mathcal{P}} \diam P $ is sufficiently small (less than $\frac 1{10}$ of the injectivity radius of $X$ would be sufficient), and consider its refinement under the time one geodesic flow.  Any partition element of the $\lfloor 2N \log{p}\rfloor$-th refinement is contained in a union of $O_c(1)$ ``tubes'' of the form $x_lB(cp^{-2N},\tau)$ for some points $x_l \in S^*M$.  For convenience, we take $c$ sufficiently small, so that the remarks following the proof of Lemma~\ref{disjoint} apply.

Positive entropy on almost every ergodic component is equivalent to the statement:  for any $\eta>0$ there exists $\delta(\eta)>0$ such that, for all $N$ sufficiently large,  any collection of distinct partition elements of the $\lfloor 2N \log{p}\rfloor$-th refinement of $\mathcal{P}$ whose union has total mass $>\eta$, must contain at least $ p^{\delta N}$ partition elements (cf. e.g. \cite{Walters-book}).
Therefore, take a collection $\{ E_1, E_2, \ldots, E_K \}$ of distinct partition elements of the $\lfloor 2N \log{p}\rfloor$-th refinement of $\mathcal{P}$, of cardinality $K$, and set $\mathcal{E}=\bigcup_{k=1}^K E_k$ to be their union; we wish to show that $K \geq p^{\delta N}$ for some $\delta(\eta)$, and all $N$ sufficiently large.

Let $1_{E_k}$ denote the characteristic function of each $E_k$, and similarly $1_{\mathcal{E}} = \sum_{E_k \subset \mathcal{E}} 1_{E_k}$.
To each $E_k$ we associate, as above, $O_c(1)$ tubes $B_{k,l} = x_{k,l}B(\epsilon,\tau)$ whose union contains $E_k$  with $\epsilon = cp^{-2N}$. Let $E_{k,l}$ denote the intersection $E_k \cap x_{k,l}B(\epsilon,\tau)$. 

Now assume that $\mu(\mathcal{E})>\eta$; since we have assumed $\mu (\partial  P) = 0$ for all $P \in \mathcal{P}$ (hence $\mu (\partial E _ k) = 0$ for all $k$), this implies that there exists a $j$ such that $\mu_j(\mathcal{E})=||\Phi_j1_\mathcal{E}||_2^2 > \eta$ as well.   Consider the correlation
$$\langle K_N (\Phi_j1_{\mathcal{E}}),  \Phi_j1_{\mathcal{E}}\rangle$$
where $K_N$ is the operator from Lemma~\ref{operator}.
We will estimate this correlation in two different ways.  
First, we can expand
\begin{equation*}
\langle K_N (\Phi_j1_{\mathcal{E}}),  \Phi_j1_{\mathcal{E}}\rangle_{L^2(S^*M)}
 =  \sum_{k,l} \sum_ {k ', l '} \langle K_N(\Phi_j1_{E _ {k, l}}), \Phi_j1_{E_ {k ', l '}}\rangle_{L^2(E_k)}
.\end{equation*}
Now lift $x_{k,l}$ and  $x_{k ', l '}$ from $X$ to elements $\underline {x _ {k, l}}$ and $ \underline {x _ {k ', l '}}$ in a compact fundamental domain $\mathcal{F}$ for $\Gamma$ in $\PSL (2, \R)$ as in the proof of Lemma~\ref{disjoint}. We also lift the sets $E _ {k, l}$ and $ E _ {k ', l '}$ to subsets $\underline {E _ {k, l}}$ and $ \underline {E _ {k ', l '}}$ of $\PSL (2, \R)$ so that $\underline {E _ {k, l}} \subset \underline {x _ {k, l}} B (\epsilon, \tau)$ and similarly for $k ', l '$. Consider all $\alpha_i \in R(p^j), j \leq N$ for which 
\[\underline {\alpha_i x_{k,l}} B(\epsilon, \tau) \cap \underline {x _ {k',l'}} B(\epsilon, \tau) \neq \emptyset;\]
 say there are $v$ such. By Lemma~\ref{disjoint} (more precisely, by the comments following the proof of this lemma), we have $v = O(N)$. By \ref{small matrix coefficients} of Lemma~\ref{operator}, for any $y \in X$
\begin{equation*}
\left| K_N(\Phi_j1_{E_{k,l}}) (y)\right| \lesssim p^{-N \delta} \sum_{j \leq N} \sum_{z \in S_{p^j}(y)}\left| \Phi _ j (z)\right| 1_{E_{k,l}}(z)
,\end{equation*}
hence (implicitly identifying between functions on $X$ and left $\Gamma$-invariant functions on $\PSL (2, \R)$)
\begin{align*}
\langle K_N(\Phi_j1_{E_{k,l}}), \Phi_j1_{E_{k ',l ' }} \rangle & \lesssim p^{-N \delta} \sum_ {i=1}^v \int_  {\underline {\alpha _ i E_{k,l}} \cap \underline {E_{k',l'}}} \left| \Phi _ j (\underline {\alpha _ i ^{-1} z})\right| \cdot \left| \Phi _ j (\underline z)\right| \,d \underline {z} \\
& \lesssim p^{-N \delta} \sum_ {i=1}^v \norm {\Phi _ j}_{L^2(E_{k,l})} \norm {\Phi _ j}_{L^2(E_{k',l'})} \\
& \lesssim p^{-N \delta}N \norm {\Phi _ j}_{L^2(E_{k,l})} \norm {\Phi _ j}_{L^2(E_{k',l'})}
.\end{align*}
It follows that
\begin{equation}\label{pointwise}
\begin{aligned}
\langle K_N (\Phi_j1_{\mathcal{E}}),  \Phi_j1_{\mathcal{E}}\rangle &\lesssim p^{-N \delta}N \left (\sum_ {k, l} \norm {\Phi _ j}_{L^2(E_{k,l})} \right) ^2 \\
& \lesssim p^{-N \delta}NK \sum_ {k,l} \norm {\Phi _ j}_{L^2(E_{k,l})} ^2 \\
& \leq p^{-N \delta}NK
\end{aligned}
.\end{equation}

On the other hand, we can decompose $\Phi_j1_\mathcal{E}$ spectrally into an orthonormal basis (of $L^2(S^*M)$) of $T_p$ eigenfunctions $\{ \psi_i \}$, which {\em a fortiori} also diagonalize $K_N$.  After applying Lemma~\ref{quasimodes}, and restricting to a subsequence if necessary, we may assume that $\Phi_j \in S_j$, the space spanned by those $\psi_i$ with $T_p$-eigenvalue in $[2 \cos(\theta)-2 \omega_j, 2 \cos(\theta)+2 \omega_j]$, where $2 \cos{\theta}$ is an approximate $T_p$-eigenvalue for all $\Phi_j$.  We denote by $\Pi_{S_j}$ the orthogonal projection to $S_j$, and observe that
$$||\Phi_j1_\mathcal{E}||_2^2 = \left|\left| \Pi_{S_j}(\Phi_j1_\mathcal{E})\right|\right|_2^2 + \sum_{\psi_i\notin S_j} |\langle \Phi_j1_\mathcal{E}, \psi_i \rangle|^2$$
Now, since $\Phi_j \in S_j$ is a unit vector, we have
\begin{eqnarray*}
\left|\left| \Pi_{S_j}(\Phi_j1_\mathcal{E})\right|\right|_2 & = & \max_{\{ \psi \in S_j : ||\psi||_2=1 \}} \langle \Phi_j1_\mathcal{E}, \psi \rangle\\
& \geq& \langle \Phi_j1_\mathcal{E}, \Phi_j \rangle = ||\Phi_j1_\mathcal{E}||_2^2
\end{eqnarray*}
and therefore
\begin{eqnarray*}
\sum_{\psi_i\notin S_j} |\langle \Phi_j1_\mathcal{E}, \psi_i \rangle|^2 & = & ||\Phi_j1_\mathcal{E}||_2^2 - \left|\left| \Pi_{S_j}(\Phi_j1_\mathcal{E})\right|\right|_2^2\\
& \leq & ||\Phi_j1_\mathcal{E}||_2^2 - ||\Phi_j1_\mathcal{E}||_2^4\\
& < & ||\Phi_j1_\mathcal{E}||_2^2 (1-\eta)
\end{eqnarray*}
by the assumption that $||\Phi_j1_\mathcal{E}||_2^2 > \eta$.

Now by Lemma~\ref{operator}, since $\{ \psi_i \}$ diagonalizes $K_N$, and the $K_N$ eigenvalue for each $\psi_i$ is at least $-1$, while the $K_N$ eigenvalue for eigenfunctions in $S_j$ is greater than $\eta^{-1}$, we have
\begin{eqnarray}
\langle K_N (\Phi_j1_{\mathcal{E}}),  \Phi_j1_{\mathcal{E}}\rangle
& = & \sum_{\psi_i}|\langle \Phi_j1_\mathcal{E}, \psi_i \rangle|^2 \langle K_N \psi_i, \psi_i \rangle\nonumber\\
& \geq & \sum_{\psi_i \in S_j} |\langle \Phi_j1_\mathcal{E}, \psi_i \rangle|^2 \langle K_N \psi_i, \psi_i \rangle - \sum_{\psi_i\notin S_j} |\langle \Phi_j1_\mathcal{E}, \psi_i \rangle|^2\nonumber\\
& > & \sum_{\psi_i \in S_j} |\langle \Phi_j1_\mathcal{E}, \psi_i \rangle|^2 \cdot \eta^{-1} - ||\Phi_j1_\mathcal{E}||_2^2 (1-\eta)\nonumber\\
& \geq & \left|\left| \Pi_{S_j}(\Phi_j1_\mathcal{E})\right|\right|_2^2 \cdot \eta^{-1} - ||\Phi_j1_\mathcal{E}||_2^2 (1-\eta)\nonumber\\
& > & ||\Phi_j1_\mathcal{E}||_2^2 (||\Phi_j1_\mathcal{E}||_2^2 \cdot \eta^{-1} - (1-\eta))\nonumber\\
& > & \eta(\eta \cdot \eta^{-1} -1 +\eta) = \eta^2 >0\label{spectral}
\end{eqnarray}
where we have used the estimate above $ \left|\left| \Pi_{S_j}(\Phi_j1_\mathcal{E})\right|\right|_2^2 \geq ||\Phi_j1_\mathcal{E}||_2^4$.

Therefore, combining (\ref{pointwise}) and (\ref{spectral}), we have
$$Np^{-\delta N} K \gtrsim \eta^2$$
and so
$$K \gtrsim \eta^2 N^{-1}p^{\delta N}$$
and there exists $\delta'(\eta)>0$ such that the right hand side is $\geq p^{-\delta'N}$ for all $N$ sufficiently large.

Since this holds for any collection of partition elements of total $\mu$-measure $>\eta$, we conclude that there is at most $\mu$-measure $\eta$ on ergodic components of entropy less than $\delta'\gtrsim \delta>0$.  Taking $\eta \to 0$, we get positive entropy on a.e. ergodic component of $\mu$.
\end{proof}

\section{Irreducible Quotients of $\mathbb{H}\times\mathbb{H}$}\label{HxH}

In this section, we let $M=\Gamma\backslash \mathbb{H}\times\mathbb{H}$, where $\Gamma$ is an irreducible, cocompact, discrete subgroup of $\PSL(2,\mathbb{R})\times \PSL(2,\mathbb{R})$.  Here we do not assume that a Hecke correspondence is available to apply our methods to, but instead we assume that our functions are joint $o(1)$-quasimodes of the two partial Laplacians (in each coordinate).  Since the eigenvalue of the Laplace operator on $M$ is the sum of the eigenvalues of the two partial Laplacians, the semiclassical limit entails at least one of the two partial eigenvalues tending to $\infty$ (perhaps after restricting to a subsequence, if necessary).  Here, we consider the case where one partial eigenvalue tends to $\infty$ while the other remains bounded; without loss of generality, we assume throughout that the approximate eigenvalue of the partial Laplacian $\Delta_1$ in the first coordinate grows, while that of $\Delta_2$ in the second coordinate is bounded.  The discussions of section~\ref{ml lifts} apply in the first coordinate (see Section~\ref{HxH lift}), and we find that any quantum limit measure on $\Gamma\backslash \PSL(2,\mathbb{R})\times\mathbb{H}$ arising from such a sequence is invariant under the diagonal subgroup acting on the first ($\PSL(2,\mathbb{R})$) coordinate.  As discussed in Lemma~\ref{hyperbolic recurrence by propagation lemma}, such a limit measure is also recurrent under translations in the second coordinate.  We then use the action of $\Delta_2$ as a replacement for the Hecke operator in applying the methods of the preceding sections to get our positive entropy result, which therefore implies QUE by \cite{Lin}.

\begin{Theorem}\label{HxH main}
Let $\phi_j$ be a sequence of $L^2$-normalized joint $o(1)$-quasi\-modes for the two partial Laplacians $\Delta_1,\Delta_2$ on $M$, such that the approximate eigenvalues of $\Delta_1$  grow to $\infty$, while those of $\Delta_2$ remain bounded.  Then any weak-* limit point $\mu$ of the microlocal lifts $\mu_j$ has positive entropy on almost every ergodic component.  
\end{Theorem}

\subsection{Our kernel for the hyperbolic plane}\label{HxH kernel}

We recall the Selberg/Harish-Chandra transform (see eg. \cite[Chapter 1.8]{Iwaniec})
\begin{eqnarray}
h(r) & = & \int_{-\infty}^\infty e^{iru}g(u)du\nonumber\\
g(u) & = & 2Q\left(\sinh^2\left(\frac{u}{2}\right)\right)\nonumber\\
k(t) & = & -\frac{1}{\pi} \int_t^\infty \frac{dQ(\omega)}{\sqrt{\omega-t}}\label{Selberg/HC}
\end{eqnarray}
relating a radial kernel $k(x,y) = k(t(x,y)) = k(\sinh^2(dist(x,y)/2))$ with its spherical transform $h(r)$, which gives the eigenvalues under convolution with $k$ for each Laplace eigenfunction $(\Delta + (\frac{1}{4}+r^2))\phi = 0$ of spectral parameter $r$.  Intuitively, the variable $u$ for $g(u)$ represents wave propagation times.

We will start with the Fourier pair
\begin{eqnarray*}
h_T(r) & = & \frac{ \cos(rT)}{\cosh(\pi r/2)}\\
g_T(\xi) & = & \frac{4 \cosh{\xi}\cosh{T}}{\cosh{2\xi} + \cosh{ 2T}}
\end{eqnarray*}
essentially the same as that used in \cite{IwaniecSarnak}, albeit in inverted roles.  Notice that for untempered $r$--- i.e., $ir \in [-\frac{1}{2}, \frac{1}{2}]$--- we have
$$h_T(r) = \frac{\cosh(irT)}{\cos{\pi i r/2}} \gtrsim 1$$
since $\cosh \geq 1$, and the argument of the $\cos$ term is in $[-\frac{\pi}{4}, \frac{\pi}{4}]$, where $\cos$ is uniformly bounded below.

We use this kernel to establish the analogue of the Propagation Lemma~\ref{propagation}:
\begin{Lemma}\label{H kernel}
Let $k_T$ be the radial kernel corresponding to the above function $h_T(r)=\frac{\cos(rT)}{\cosh(\pi r)}$.  Then
\begin{itemize}
\item{$||k_T||_\infty \lesssim e^{-T/2}$}
\item{$k_T$ decays rapidly outside of a ball of radius $4T$; in fact, 
$$\int_{t=\sinh^2(2T)}^\infty |k_T(t)| dt  \lesssim e^{-T}$$}
\end{itemize}
\end{Lemma}
This Lemma is analogous to Lemma~\ref{propagation}, and can be understood in terms of propagation of the hyperbolic wave equation.  Our proof here, estimating the explicit transform directly, is morally the same as that of Lemma~\ref{propagation}, though perhaps more direct at the expense of transparency.

\begin{proof}
We have
\begin{eqnarray*}
2 Q(\sinh^2(\xi/2)) & = &  4 \frac{2\cosh{{\xi}}\cosh{{T}}}{2\cosh{2\xi} + 2\cosh{ 2T}}\\
Q(\omega) & = &  2\frac{(4\omega+2)\cosh{T}}{((4\omega+2)^2-2)+2\cosh{2T}}
\end{eqnarray*}
which implies that
\begin{eqnarray}
Q'(\omega) & = & (8\cosh{T})\frac{(4\omega+2)^2-2+2\cosh{2T} - 2(4\omega+2)^2}{[(4\omega+2)^2-2+2\cosh{2T}]^2}\nonumber\\
& = & (8\cosh{T}) \frac{2\cosh{2T}-2 -(4\omega+2)^2}{[2\cosh{2T}-2 + (4\omega+2)^2]^2}\nonumber\\
& \lesssim_C & \left\{ \begin{array}{ccc} \omega^{-2}\cosh{T} & \quad \quad & (4\omega+2)^2\geq \frac{1}{C}(2\cosh{2T}-2)\\ \cosh{T}(\cosh{2T})^{-1} & \quad\quad & (4\omega+2)^2\leq C(2\cosh{2T}-2) \end{array} \right.\label{Q'}
\end{eqnarray}

In particular, if $(4t+2)^2\geq 2\cosh(2T)-2$, then 
\begin{eqnarray*}
|\pi k_T(t)| & = & \left|\int_0^\infty v^{-1/2}Q'(v+t) dv\right|\\
& \lesssim & \cosh{T} \int_0^\infty v^{-1/2}(v+t)^{-2} dv\\
& \lesssim & \cosh{T}\left(t^{-2}\int_0^t v^{-1/2}dv + t^{-1}\int_t^\infty v^{-3/2}dv\right)\\
& \lesssim & \cosh{T} (t^{-2}t^{1/2} + t^{-1}t^{-1/2} ) \lesssim \cosh{T}\cdot t^{-3/2}
\end{eqnarray*}
and therefore, since $t\geq \sinh^2(2T)$ implies $(4t+2)^2\geq 2\cosh(2T)-2$, we have by (\ref{Q'})
\begin{eqnarray*}
\int_{t=\sinh^2(2T)}^\infty |k_T(t)| dt & \lesssim & \cosh{T}\int_{t=\sinh^2(2T)}^\infty t^{-3/2}dt\\
& \lesssim & \cosh{T} \cdot (\sinh(2T))^{-1}\\
& \lesssim & e^{-T}
\end{eqnarray*}
since $ \cosh{T}\lesssim e^T$ and $\sinh(2T)\gtrsim e^{2T}$ for $T\geq 1$, say.

For the $||k_T||_\infty$ statement, we choose an appropriate value of $C$ simplifying (\ref{Q'}) to
$$|Q'(\omega)| \lesssim  \left\{ \begin{array}{ccc} \omega^{-2}\cosh{T} & \quad \quad & \omega\geq \cosh{T} \\ \cosh{T}(\cosh{2T})^{-1} & \quad\quad & \omega \leq 2\cosh{T} \end{array} \right.$$
Apply this first in the case $t\geq \cosh{T}$, to get as in the previous estimate
\begin{eqnarray}
|\pi k_T(t)| & \lesssim & \cosh{T}(t^{-3/2})\label{large t}\\
& \lesssim & \cosh{T}\cdot (\cosh{T}^{-3/2}) \lesssim e^{-T/2}\nonumber
\end{eqnarray}
Finally, if $t\leq \cosh{T}$, we estimate 
\begin{eqnarray}
|\pi k_T(t)| & \lesssim & \int_0^{\cosh{T}} v^{-1/2} \cosh{T}(\cosh{2T})^{-1}dv + \int_{\cosh{T}}^\infty v^{-1/2}(v+t)^{-2}\cosh{T}dv\nonumber\\
& \lesssim & (\cosh{T})^{3/2}(\cosh{2T}^{-1}) + \cosh{T}\int_{\cosh{T}}^\infty v^{-5/2}dv\nonumber\\
& \lesssim & (\cosh{T})^{-1/2} + (\cosh{T})^{-1/2} \nonumber\\
& & \label{small t}
\end{eqnarray}
by estimating $|Q'(v+t)|$ separately for $(v+t)\leq 2\cosh{T}$ and $(v+t)\geq \cosh{T}$.  
\end{proof}

We now use this family of kernels in a construction analogous to Lemma~\ref{operator}:
\begin{Lemma}\label{HxH operator}
Let $0<\eta<\frac{1}{2}$, and $r_j^{(2)}$ an approximate spectral parameter.  For any sufficiently large $N\in\mathbb{N}$ (depending on $\eta$), there exists an operator $K_N$ on $\Gamma\backslash SL(2,\mathbb{R})\times\mathbb{H}$ satisfying:
\begin{enumerate}
\item{$K_N$ is given by convolution with a kernel $k_N$ in the second coordinate, supported in the ball of radius $2N$; i.e.
$$K_N(f)(x,z) = \int_{d(w,z)\leq 2N} f(x,w)k_N(w)dw$$ }
\item{We have the estimate
$||k_N||_\infty \lesssim e^{-\delta N}$ for some $\delta(\eta)>0$.  In fact, we can choose $\delta = \frac{1}{2}\eta^2$.}
\item{The spherical transform $h_{k_N}$ is uniformly bounded below, and $h_{k_N}(r)\gtrsim \eta^{-1}$ for 
all $|r-r_j^{(2)}| < 1/2N$.}
\end{enumerate}
\end{Lemma}

\begin{proof}

We first use the fact that $\int |k_T(t)| dt$ is small outside the ball of radius ${4T}$ to get a modified kernel $\tilde{k}_T$ cut off to be supported inside this ball.  Namely, define
$$\tilde{k}_T(t) =\left\{ 	\begin{array}{ccc} k_T(t) 	& \quad\quad & t\leq \sinh^2(2T)\\ 0 & \quad\quad & t > \sinh^2(2T)	\end{array} \right.$$
and by integrating against a spherical eigenfunction (see eg. \cite[1.7]{Iwaniec}), we see that $k_T - \tilde{k}_T$ has spherical transform bounded by
$$\int_{w\geq 4T} |k(\sinh^2(w))|dw \lesssim e^{-T}$$
since spherical eigenfunctions decay away from the origin.  Therefore $\tilde{k}_T$ has spherical transform $\tilde{h}_T$ satisfying
$$\left|\tilde{h}_T(r) - \frac{\cos(rT)}{\cosh{\pi r/2}} \right| \lesssim e^{-T}$$
We now set $L:=\lceil \eta^{-1}\rceil$, and take the linear combination 
$$h_{L,T} (r):= \sum_{j=1}^{2L} \frac{2L-j}{L} \tilde{h}_{2jT}(r) $$
which satisfies 
$$\left|	h_{L,T}(r) - \frac{F_{2L}(2Tr) - 1}{\cosh \pi r/2}	\right| \lesssim \sum_{j=1}^{2L} e^{-2jT} < 1$$
 so that $h_{L,T}$ is everywhere uniformly bounded below, and $h_{L,T}(r)>L$ for untempered $r\in i\mathbb{R}$, in analogy with section~\ref{propagator}.

Of course, if our approximate spectral parameter $r^{(2)}_j$ is tempered (i.e., $r^{(2)}_j \in\mathbb{R}$), then we would like to pick $T$ in such a way that $F_{2L}(2Tr)$ is large at $r^{(2)}_j$, and we invoke Dirichlet's Theorem once again to find a suitable such $T$.

  Suppose $r^{(2)}_j\in\mathbb{R}$ is tempered, and set $\theta_0 = r^{(2)}_j \bmod{2\pi}$.  Take $N$ from the hypotheses of the Lemma, which represents the time to which we will allow propagation, and let  $Q=\lceil \frac{1}{8}N\eta\rceil$.  Apply the argument from the proof of Lemma~\ref{operator} to find $q'\in 2\mathbb{Z}$ such that $N\eta^2 \lesssim q' \leq 2Q$, and satisfying $F_{2L}(q'\theta) >L+1$ for all  $\theta = r\bmod{2\pi}$ where $|r- r^{(2)}_j| <1/2N$.    Now set $2T=q'$, and we have 
$$h_{L,T}(r) > \frac{F_{2L}(q'r)-1}{\cosh{\pi r/2}} - O(1) \gtrsim_{r_j^{(2)}}  L$$
for all $|r- r^{(2)}_j| <1/2N$.  Note that our estimate depends on $r^{(2)}_j$, but we have assumed that $r_j^{(2)}$ remains bounded, so that $h_{L,T}(r) \gtrsim L$, with the implied constant depending on our uniform bound for $r_j^{(2)}$.

Finally, we define $k_{L,T}:= \sum_{j=1}^{2L} \frac{2L-j}{L}\tilde{k}_{2jT}$ to be the kernel corresponding to $h_{L,T}$, and note that since $||\tilde{k}_{2jT}||_\infty = ||k_{2jT}||_\infty$, we have
$$||k_{L,T}||_\infty \leq \sum_{j=1}^{2L}\left|\left| k_{2jT}\right|\right|_\infty \lesssim \sum_{j=1}^{2L} e^{-(2jT)/2} \lesssim e^{-2T}$$
decays exponentially in $T$.

In conclusion, we have constructed a radial kernel $k_{L,T}$ with the following properties:
\begin{itemize}
\item{$k_{L,T}$ is supported in the ball of radius $4(2L\cdot 2T)<2N$.}
\item{$||k_{L,T}||_\infty \lesssim e^{-T} < e^{-\frac{1}{2}N\eta^2}$ for $N$ sufficiently large.}
\item{The spherical transform of $k_{L,T}$ is uniformly bounded below, and is $\gtrsim \eta^{-1}$ at all spectral components in $[r_j^{(2)}-\frac{1}{2N}, r_j^{(2)}+\frac{1}{2N}]$.}
\end{itemize}
as required.

\end{proof}

Armed with this kernel of Lemma~\ref{HxH operator}, the proof of Theorem~\ref{HxH main} proceeds exactly as in section~\ref{Hecke}; we are missing only the analogue of Lemma~\ref{disjoint}:
\begin{Lemma}\label{HxH disjoint}
There are $c,\kappa>0$ so that for  any $(x,z)$ in a compact fundamental domain $\mathcal{F} <  \PSL(2,\mathbb{R})\times \PSL(2,\mathbb{R})$  with respect to the action of $\Gamma$ on the left, there are at most $cN$ distinct $\gamma \in \Gamma$ such that there exist $g_1, g_2 \in \PSL(2,\mathbb{R})$ satisfying
\begin{equation}\label{gamma condition}
(x, zg_1) \in \gamma (xB(e^{-\kappa N}, \tau), zg_2) \qquad \log\|g_1\|, \log\|g_2\| <N
\end{equation}
\end{Lemma}

Recall that $\Gamma$ is an irreducible, cocompact lattice in $\PSL (2, \R) \times \PSL (2, \R)$. In this context, irreducibility is equivalent to $\Gamma$ intersecting the group $\left\{ e \right\} \times \PSL (2, \R)$ trivially \cite{Raghunathan}.
A complete classification of such (up to commensurability) can be obtained from Margulis' Arithmeticity Theorem \cite{Margulis-book}. This classification (and Liouville type bounds on the approximation of algebraic numbers) in particular implies a more quantitative form of the triviality of $\Gamma \cap \left\{ e \right\} \times \PSL (2, \R)$, namely that there are $C, \kappa>0$ so that
\begin{equation} \label{property of irreducible lattices}
\norm {g_1 - e} \geq \kappa \norm {g_2}^{-C} \qquad \text {for any $(g_1,g_2) \in \Gamma \setminus \left\{ e \right\}$}
.\end{equation}
We shall also make use of the following estimate:

\begin{Proposition}\label{abelian subgroups are small}
Let $\Gamma$ be an irreducible cocompact lattice in $\PSL (2, \R) \times \PSL (2, \R)$. Then
there is a $C _ 1$ so that for every \emph{abelian} subgroup $H < \Gamma$
and $N>1$
\begin{equation*}
\absolute {\left\{ h =(h', h'') \in H: \norm {h'}\leq 100, \norm {h ''}\leq N \right\}} \leq C _ 1 \log N 
.\end{equation*}
\end{Proposition}

\begin{proof}
Let $\mathcal{F}$ be a compact fundamental domain for $G = \PSL (2, \R) \times \PSL (2, \R)$ with respect to $\Gamma$. Then
\begin{equation*}
\left\{ g (\Gamma \setminus \left\{ e \right\}) g ^{-1}: g \in G \right\} = \left\{ g (\Gamma \setminus \left\{ e \right\}) g ^{-1}: g \in \mathcal{F} \right\}
\end{equation*}
is a closed subset of $G$ not containing $e$. It follows that there is a $\delta > 0$ so that for every $\gamma \in \Gamma \setminus \left\{ e \right\}$
\begin{equation*}
\inf_ {g \in G} \norm {g \gamma g ^{-1} - e} > \delta
.\end{equation*}
Let now $H$ be an abelian subgroup of $G$. Since $\Gamma$ has a finite index torsion free subgroup, if $H$ contains only torsion elements then the order of $H$ is $O (1)$. Let now $\gamma \in H$ be an element of infinite order. The centralizer $C _ G (\gamma)$ of $\gamma$ is conjugate to either $T = K \times A, A \times K, A \times A$ with $K = \SO (2, \R) /\{ \pm 1 \}$ and $A < \PSL (2, \R)$ the diagonal group (if the centralizer is $K \times K$ then since $\Gamma$ is discrete we must have that $\gamma$ is a finite order). Both $K$ and $A$ have the properties that for any $h$ in either group $\norm {h-e} \geq c \inf_ {g \in \PSL (2, \R)} \norm {g h g ^{-1} - e}$ (with the standard choice of norm, we may even take $c = 1$).

It follows that if $g _ 0$ is chosen so that $g _ 0 H g _ 0 ^{-1} \leq T$
then
\begin{multline*}
g_0\left(H \cap \left\{ (h ', h ''): \norm {h '} \leq 100, \norm {h ''} \leq  N \right\}\right)g_0^{-1} \subset\\ 
 T \cap \left\{ (h ', h ''): \norm {h '} \leq 100/c, \norm {h ''} \leq N /c  \right\}
\end{multline*}
hence by the pigeonhole principle if this intersection is of cardinality $\geq C_1 \log N$ with sufficiently large $C _ 1$ there will be two distinct $\gamma _ 1, \gamma _ 2 \in H$ so that $\norm {g_0\gamma _ 1 \gamma _2 ^{-1}g_0^{-1} - e} < \delta$, in contradiction to the definition of~$\delta$.
\end{proof}

\begin{proof}[Proof of Lemma~\ref{HxH disjoint}]
As in the proof of Lemma~\ref{disjoint} a key observation is that if $\gamma_1, \gamma_2 \in \Gamma$ are such that there exist $g_i, g ' _ i, h_i \in G$ ($i=1,2$) with $\| g_i \|,\| g ' _i \| \leq e^N$ and $h _ i \in B (e ^ {- \kappa N}, \tau)$ satisfying
\begin{equation} \label{gamma property} (x,zg_i) = \gamma _ i (x h _ i, z g ' _i)
\end{equation}
then $\gamma_1$ and $\gamma_2$ commute.
Indeed, letting $\pi _ i$ denote the projection from $\PSL (2, \R) \times \PSL (2, \R)$ to the $i$ component ($i = 1, 2$) we have that
\begin{align*}
\|\pi _ 1 [\gamma _ 1, \gamma _ 2]-e\|&= \|\pi _ 1 (\gamma _ 1 \gamma _ 2 \gamma _ 1 ^{-1} \gamma _ 2 ^{-1})-e\|\\
& = \|x h _ 1 ^{-1} h _ 2 ^{-1} h _ 1 h _ 2 x ^{-1}-e\|= O(e ^ {- \kappa N})
\end{align*}
with an implicit constant depending only on $\tau$, while
\begin{equation*}
\norm {\pi _ 2 [\gamma _ 1, \gamma _ 2]} = \norm {zg_1 {g_1'}^{-1} g_2{g_2'}^{-1} {g_1'} g_1^{-1} g_2' g_2^{-1} z^{-1}} = O(e ^ {8N})
\end{equation*}
since for $g \in \SL(2,\R)$ the norms of $g$ and $g^{-1}$ are comparable.
Hence by \eqref{property of irreducible lattices} if $\kappa/8>C$ and $N$ is sufficiently large we may conclude that $[\gamma _ 1, \gamma _ 2] = 1$.

The lemma now follows by applying Proposition~\ref{abelian subgroups are small} to the set of all $\gamma_i$'s satisfying \eqref{gamma property}.
\end{proof}

\subsection{Recurrence for Joint Quasimodes on $\mathbb{H}\times\mathbb{H}$}

Here, we prove the analogue of Lemma~\ref{Hecke recurrence by propagation lemma} for our joint quasimodes on $\Gamma\backslash\mathbb{H}\times\mathbb{H}$.

\begin{Lemma}\label{hyperbolic recurrence by propagation lemma}
Let $\Phi_j $ be a sequence of $\omega_j$-quasimodes for $\Delta_2$ (the Laplacian operator on the second component) on  $\Gamma\backslash \PSL(2,\mathbb{R})\times\mathbb{H}$, with  $\omega_j\to 0$. Suppose that $|\Phi_j|^2d\vol$ converges weak-* to a measure $\mu$.  Consider a point $(x,z)\in \Gamma\backslash \PSL(2,\mathbb{R})\times\mathbb{H}$, and let $B$ be a small open ball around $e \in PSL(2,\mathbb{R})$, so that $xB\times \{z\}$ is a small ball in the first coordinate around $(x,z)$.  Then for every $(x, z) \in \supp \mu$ and $\eta \in (0,1)$,
$$\liminf_{j \to \infty} \frac{\eta^2\int_{x {B} \times \{ w : d(z,w)\leq N \}} |\Phi_j(x' , w)|^2 dx' dw}{\int_{x {B} \times \{ w : d(z,w)\leq \eta\}} |\Phi_j(x' , w)|^2 dx'  dw}\to \infty \qquad \text{as $N \to \infty$} $$
uniformly in $(x,z), \eta$, and the radius of $B$.
\end{Lemma}

\begin{proof}
We proceed as in Lemma~\ref{Hecke recurrence by propagation lemma}, using a simplified version of the construction in section~\ref{HxH kernel}.  Let $L$ be a large integer, and $r^*$ an approximate spectral parameter for $\Phi_j$ for $\Delta_2$.  Find a $q \in \{1, \ldots, 100L\}$ so that
\begin{equation}\label{HxH choice of q}
|qr^* \bmod{2\pi}| \leq \frac{\pi}{50L}
\end{equation}
and set
$$k = \sum_{l=1}^L \tilde{k}_{2ql}$$
where $\tilde{k}_{2ql}$ is given by the spherical kernel of Lemma~\ref{H kernel}, cut off at radius $8ql$, as in section~\ref{HxH kernel}.  The spherical transform of $k$ therefore satisfies
$$\left| h_k(r) - \sum_{l=1}^L \frac{\cos(rl)}{\cosh(\pi r/2)}\right| \lesssim 1$$
for all $r$.
We computed in (\ref{small t}) and (\ref{large t}) that 
$$|\tilde{k}_{2lq}(t)| \leq | k_{2lq}(t)| \lesssim \left\{ 	\begin{array}{ccc}	(\cosh{2lq})^{-1/2} & \quad\quad & t\leq \cosh{2lq}\\ \cosh{2lq}\cdot t^{-3/2} & \quad\quad & t\geq \cosh{2lq}	\end{array}	\right.$$
so that for $\cosh{2lq} < t \leq \cosh{2(l+1)q} $ we have
\begin{eqnarray*}
k(t) & \lesssim & \sum_{1\leq l'\leq l} \cosh{2l'q}\cdot t^{-3/2} + \sum_{l+1\leq l'\leq L} (\cosh{2l'q})^{-1/2}\\
& \lesssim & \cosh{2lq}\cdot t^{-3/2} + (\cosh{2(l+1)q})^{-1/2}
\end{eqnarray*}

We wish to estimate
$$|(k \ast \Phi_j)(x,z)| = \left|\int_\mathbb{H} k(\sinh^2(d(z,w)/2)) \Phi_j(x,w)dw\right|$$
by dividing into the following annuli:
\begin{eqnarray*}
A_0 & = & \{ w : \sinh^2(d(z,w)/2) \leq \cosh{2q}\}\\
A_l & = & \{ w : \cosh{2lq}< \sinh^2(d(z,w)/2)\leq \cosh{2(l+1)q} \} \\
&&\qquad\qquad\qquad\qquad\qquad\qquad \qquad l=1,2,\ldots, L-1\\
A_L & = & \{ w : \cosh{2Lq} < \sinh^2(d(z,w)/2) \leq \sinh^2(4qL)\}
\end{eqnarray*}
which cover the support of $k$.
Applying Cauchy-Schwarz, we have
\begin{eqnarray*}
\lefteqn{|(k \ast \Phi_j)(x,z)| ^2}\\
& \lesssim & \left|\sum_{l=0}^L 	\int_{A_l} k (\sinh^2(d(z,w)/2))\Phi_j(x,w) dw	\right| ^2 \\
&\lesssim & L \sum_{l=0}^L  \left|\int_{A_l} k(\sinh^2(d(z,w)/2))\Phi_j(x,w) 	dw\right| ^2 \\
& \lesssim & L  \sum_{l=0}^{L} \left(\int_{A_l} |k(\sinh^2(d(z,w)/2))|^2dw  \int_{A_l}\left|\Phi_j(x,w) 	\right|^2 dw\right)
\end{eqnarray*}

Now recall that in hyperbolic polar coordinates around $z$, the infinitesimal hyperbolic area element is given by $d \theta dt$ with $t =\sinh ^2 (d (z, w) / 2)$. Thus for $1 \leq l \leq L-1$ we have the upper bound
\begin{eqnarray*}
\lefteqn{\int_{A_l} |k(\sinh^2(d(z,w)/2))|^2 dw}\\
& \lesssim & \int_{t=\cosh{2lq}}^{\cosh{2(l+1)q}} \Big(\cosh^2{2lq}\cdot t^{-3} + (\cosh{2(l+1)q})^{-1} \Big)dt\\
& \lesssim & \cosh^2{2lq}\int_{t \geq \cosh{2lq}} t^{-3} dt + (\cosh{2(l+1)q})^{-1}\int_{t \leq \cosh{2(l+1)q}} dt\\
& \lesssim & \cosh^2{2lq} (\cosh{2lq})^{-2} + (\cosh{2(l+1)q})^{-1} (\cosh{2(l+1)q}) \lesssim 1.
\end{eqnarray*}
Moreover, we similarly estimate the integral over $A_0$ and $A_L$ by
\begin{eqnarray*}
\int_{A_0} |k(\sinh^2(d(z,w)/2))|^2 dw & \lesssim & \int_{t=0}^{\cosh{2q}} (\cosh{2q})^{-1}dt\\
& \lesssim & 1\\
\int_{A_L} |k(\sinh^2(d(z,w)/2))|^2 dw & \lesssim & (\cosh{2(L-1)q})^2 \int_{t>\cosh{2(L-1)q}} t^{-3}dt\\
& \lesssim & 1.
\end{eqnarray*}

Plugging these back into our Cauchy-Schwarz estimate, we get
\begin{eqnarray}
|(k \ast \Phi_j)(x,z)| & \lesssim & L^{1/2}\left(\sum_{l=0}^{L} \int_{A_l}\left|\Phi_j(x,w) 	\right|^2 dw\right)^{1/2} \nonumber\\
& \lesssim & \left(L \int_{d(z,w)\leq 4Lq} |\Phi_j(x,w)|^2 dw\right)^{1/2} \nonumber\\
&\leq &
\left(L \int_{d(z,w)\leq 400L ^2} |\Phi_j(x,w)|^2 dw\right)^{1/2} \nonumber
\end{eqnarray}
and similarly for every $g \in B$
$$|(k \ast \Phi_j)(xg,z)| \lesssim \left(L \int_{d(z,w)\leq 400L ^2} |\Phi_j(xg,w)|^2 dw\right)^{1/2}$$

Now by considering separately the tempered ($r^*\in\mathbb{R}$) and untempered ($-\frac{1}{2}\leq ir^*\leq \frac{1}{2}$) cases, using in the former case (\ref{HxH choice of q}), it can be verified that
$$h(r^*) = \sum_{l=1}^L \tilde{h}_{2ql}(r^*) \gtrsim L$$
Since $\Phi_j$ is is an $\omega_j$-quasimode for $\Delta_2$
$$||k \ast \Phi_j - h(r^*)\Phi_j|| = O_L(\omega_j),$$
and it follows that
\begin{eqnarray*}
\lefteqn{h(r^*)^2 \int_ {xB \times \{ w: d(z,w)<\eta \} } |\Phi_j|^2} \\ & \leq & \int_{xB \times \{ w: d(z,w)<\eta \} } |k\ast \Phi_j|^2 + O_{L,B,\eta}(\omega_j)\\
& \lesssim & L \int_{xB \times \{ d(z,w)\leq \eta \}} \int_ {\{ d(w',w)< 400L ^2 \} } |\Phi_j(y,w')|^2dwdw'dy + O_{L,B,\eta}(\omega_j) \\
&\lesssim & L \eta ^2 \int_{xB \times \{ w: d(z,w)<400 L ^2 + 1 \} } |\Phi_j|^2 + O_{L,B,\eta}(\omega_j)
\end{eqnarray*}
Since $h(r^*)^2 \gtrsim L^2$, $\omega_j \to 0$, and $(x, z) \in \supp \mu$, we conclude that
$$\liminf_{j \to \infty} \frac{\eta ^2 \int_ { xB \times \{ d(z,w)\leq 400L ^2 +1 \}}  |\Phi_j|^2}{\int_ { xB \times \{ d(z,w)\leq \eta \}} |\Phi_j|^2} \gtrsim L.$$
\end{proof}

\def\cprime{$'$}

\bibliographystyle{alpha}
\bibliography{my}

\begin{bibdiv}
\begin{biblist}

\bib{AKN}{misc}{
      author={Anantharaman, Nalini},
      author={Koch, Herbert},
      author={Nonnenmacher, St{\'e}phane},
       title={Entropy of eigenfunctions},
        date={2009},
         url={http://arxiv.org/abs/0704.1564},
        note={to appear in Proceedings of ICMP 2006},
}

\bib{ANb}{article}{
      author={Anantharaman, Nalini},
      author={Nonnenmacher, St{\'e}phane},
       title={Entropy of semiclassical measures of the {W}alsh-quantized
  baker's map},
        date={2007},
        ISSN={1424-0637},
     journal={Ann. Henri Poincar\'e},
      volume={8},
      number={1},
       pages={37\ndash 74},
      review={\MR{MR2299192 (2008f:81086)}},
}

\bib{AN}{article}{
      author={Anantharaman, Nalini},
      author={Nonnenmacher, St{\'e}phane},
       title={Half-delocalization of eigenfunctions for the {L}aplacian on an
  {A}nosov manifold},
        date={2007},
        ISSN={0373-0956},
     journal={Ann. Inst. Fourier (Grenoble)},
      volume={57},
      number={7},
       pages={2465\ndash 2523},
        note={Festival Yves Colin de Verdi{\`e}re},
      review={\MR{MR2394549}},
}

\bib{Anan}{article}{
      author={Anantharaman, Nalini},
       title={Entropy and the localization of eigenfunctions},
        date={2008},
        ISSN={0003-486X},
     journal={Ann. of Math. (2)},
      volume={168},
      number={2},
       pages={435\ndash 475},
      review={\MR{MR2434883}},
}

\bib{AnanSilberman}{article}{
      author={{Anantharaman}, N.},
      author={{Silberman}, L.},
       title={{A Haar component for quantum limits on locally symmetric
  spaces}},
        date={2010},
     journal={ArXiv e-prints},
      eprint={1009.4927},
}

\bib{BorLin}{article}{
      author={Bourgain, Jean},
      author={Lindenstrauss, Elon},
       title={Entropy of quantum limits},
        date={2003},
        ISSN={0010-3616},
     journal={Comm. Math. Phys.},
      volume={233},
      number={1},
       pages={153\ndash 171},
      review={\MR{MR1957735 (2004c:11076)}},
}

\bib{meElon}{misc}{
      author={Brooks, Shimon},
      author={Lindenstrauss, Elon},
       title={Non-localization of eigenfunctions on large regular graphs},
        date={2011},
         url={http://arxiv.org/abs/0912.3239},
        note={to appear in Israel Jour. Math.},
}

\bib{localized_example}{misc}{
      author={Brooks, Shimon},
       title={Partially localized quasimodes in large subspaces},
        date={2011},
        note={preprint},
}

\bib{Eichler}{book}{
      author={Eichler, M.},
       title={Lectures on modular correspondences},
   publisher={Tata Institute of Fundamental Research},
        date={1965},
      volume={9},
        note={Notes by S. Rangachari},
}

\bib{EKL}{article}{
      author={Einsiedler, Manfred},
      author={Katok, Anatole},
      author={Lindenstrauss, Elon},
       title={Invariant measures and the set of exceptions to {L}ittlewood's
  conjecture},
        date={2006},
        ISSN={0003-486X},
     journal={Ann. of Math. (2)},
      volume={164},
      number={2},
       pages={513\ndash 560},
      review={\MR{MR2247967}},
}

\bib{FNDB}{article}{
      author={Faure, Fr{\'e}d{\'e}ric},
      author={Nonnenmacher, St{\'e}phane},
      author={De~Bi{\`e}vre, Stephan},
       title={Scarred eigenstates for quantum cat maps of minimal periods},
        date={2003},
        ISSN={0010-3616},
     journal={Comm. Math. Phys.},
      volume={239},
      number={3},
       pages={449\ndash 492},
      review={\MR{MR2000926 (2005a:81076)}},
}

\bib{Holowinsky-Sound}{article}{
      author={Holowinsky, Roman},
      author={Soundararajan, Kannan},
       title={Mass equidistribution for {H}ecke eigenforms},
        date={2010},
        ISSN={0003-486X},
     journal={Ann. of Math. (2)},
      volume={172},
      number={2},
       pages={1517\ndash 1528},
      review={\MR{2680499 (2011i:11061)}},
}

\bib{IwaniecSarnak}{article}{
      author={Iwaniec, H.},
      author={Sarnak, P.},
       title={{\(L^\infty\)} norms of eigenfunctions of arithmetic surfaces},
        date={1995},
        ISSN={0003486X},
     journal={The Annals of Mathematics},
      volume={141},
      number={2},
       pages={301\ndash 320},
         url={http://www.jstor.org/stable/2118522},
}

\bib{Iwaniec}{book}{
      author={Iwaniec, Henryk},
       title={Spectral methods of automorphic forms},
     edition={Second},
      series={Graduate Studies in Mathematics},
   publisher={American Mathematical Society},
     address={Providence, RI},
        date={2002},
      volume={53},
        ISBN={0-8218-3160-7},
      review={\MR{1942691 (2003k:11085)}},
}

\bib{KelPert}{article}{
      author={Kelmer, Dubi},
       title={Scarring on invariant manifolds for perturbed quantized
  hyperbolic toral automorphisms},
        date={2007},
        ISSN={0010-3616},
     journal={Comm. Math. Phys.},
      volume={276},
      number={2},
       pages={381\ndash 395},
         url={http://dx.doi.org/10.1007/s00220-007-0331-2},
      review={\MR{2346394 (2009b:81067)}},
}

\bib{LinHxH}{article}{
      author={Lindenstrauss, Elon},
       title={On quantum unique ergodicity for {$\Gamma\backslash\mathbb{H}
  \times\mathbb{H}$}},
        date={2001},
        ISSN={1073-7928},
     journal={Internat. Math. Res. Notices},
      number={17},
       pages={913\ndash 933},
      review={\MR{MR1859345 (2002k:11076)}},
}

\bib{Lin}{article}{
      author={Lindenstrauss, Elon},
       title={Invariant measures and arithmetic quantum unique ergodicity},
        date={2006},
        ISSN={0003-486X},
     journal={Ann. of Math. (2)},
      volume={163},
      number={1},
       pages={165\ndash 219},
      review={\MR{MR2195133 (2007b:11072)}},
}

\bib{Margulis-book}{book}{
      author={Margulis, G.~A.},
       title={Discrete subgroups of semisimple {L}ie groups},
      series={Ergebnisse der Mathematik und ihrer Grenzgebiete (3) [Results in
  Mathematics and Related Areas (3)]},
   publisher={Springer-Verlag},
     address={Berlin},
        date={1991},
      volume={17},
        ISBN={3-540-12179-X},
      review={\MR{1090825 (92h:22021)}},
}

\bib{Raghunathan}{book}{
      author={Raghunathan, M.~S.},
       title={Discrete subgroups of {L}ie groups},
   publisher={Springer-Verlag},
     address={New York},
        date={1972},
        note={Ergebnisse der Mathematik und ihrer Grenzgebiete, Band 68},
      review={\MR{0507234 (58 \#22394a)}},
}

\bib{Riv2}{article}{
      author={Rivi{\`e}re, Gabriel},
       title={Entropy of semiclassical measures for nonpositively curved
  surfaces},
        date={2010},
        ISSN={1424-0637},
     journal={Ann. Henri Poincar\'e},
      volume={11},
      number={6},
       pages={1085\ndash 1116},
         url={http://dx.doi.org/10.1007/s00023-010-0055-2},
      review={\MR{2737492}},
}

\bib{Riv1}{article}{
      author={Rivi{\`e}re, Gabriel},
       title={Entropy of semiclassical measures in dimension 2},
        date={2010},
        ISSN={0012-7094},
     journal={Duke Math. J.},
      volume={155},
      number={2},
       pages={271\ndash 336},
         url={http://dx.doi.org/10.1215/00127094-2010-056},
      review={\MR{2736167}},
}

\bib{RS}{article}{
      author={Rudnick, Ze{\'e}v},
      author={Sarnak, Peter},
       title={The behaviour of eigenstates of arithmetic hyperbolic manifolds},
        date={1994},
        ISSN={0010-3616},
     journal={Comm. Math. Phys.},
      volume={161},
      number={1},
       pages={195\ndash 213},
      review={\MR{MR1266075 (95m:11052)}},
}

\bib{SarnakHyperbolic}{article}{
      author={Sarnak, Peter},
       title={Spectra of hyperbolic surfaces},
        date={2003},
        ISSN={0273-0979},
     journal={Bull. Amer. Math. Soc. (N.S.)},
      volume={40},
      number={4},
       pages={441\ndash 478 (electronic)},
         url={http://dx.doi.org/10.1090/S0273-0979-03-00991-1},
      review={\MR{1997348 (2004f:11107)}},
}

\bib{SarnakProgress}{article}{
      author={Sarnak, Peter},
       title={Recent progress on the quantum unique ergodicity conjecture},
        date={2011},
        ISSN={0273-0979},
     journal={Bull. Amer. Math. Soc. (N.S.)},
      volume={48},
      number={2},
       pages={211\ndash 228},
         url={http://dx.doi.org/10.1090/S0273-0979-2011-01323-4},
      review={\MR{2774090}},
}

\bib{Sound}{article}{
      author={Soundararajan, Kannan},
       title={Quantum unique ergodicity for
  $sl(2,\mathbb{Z})\backslash\mathbb{H}$},
        date={2010},
     journal={Ann. of Math. (2)},
      volume={172},
      number={2},
       pages={1529\ndash 1538},
}

\bib{Lior-Akshay}{article}{
      author={Silberman, Lior},
      author={Venkatesh, Akshay},
       title={On quantum unique ergodicity for locally symmetric spaces},
        date={2007},
        ISSN={1016-443X},
     journal={Geom. Funct. Anal.},
      volume={17},
      number={3},
       pages={960\ndash 998},
      review={\MR{MR2346281 (2009a:81072)}},
}

\bib{Lior-Akshay2}{misc}{
      author={Silberman, Lior},
      author={Venkatesh, Akshay},
       title={Entropy bounds for hecke eigenfunctions on division algebras},
        date={2010},
        note={to appear in GAFA},
}

\bib{Walters-book}{book}{
      author={Walters, Peter},
       title={An introduction to ergodic theory},
      series={Graduate Texts in Mathematics},
   publisher={Springer-Verlag},
     address={New York},
        date={1982},
      volume={79},
        ISBN={0-387-90599-5},
      review={\MR{648108 (84e:28017)}},
}

\bib{Wolpert}{article}{
      author={Wolpert, Scott~A.},
       title={Semiclassical limits for the hyperbolic plane},
        date={2001},
        ISSN={0012-7094},
     journal={Duke Math. J.},
      volume={108},
      number={3},
       pages={449\ndash 509},
      review={\MR{MR1838659 (2003b:11051)}},
}

\end{biblist}
\end{bibdiv}

\end{document}